\documentclass[11pt]{amsart}
\usepackage{multirow,bigdelim}
\usepackage{color}
\usepackage[all,cmtip]{xy}
\usepackage{amsfonts}
\usepackage{dsfont}
\usepackage{amsmath}
\usepackage{extarrows}
\usepackage{mathrsfs}
\usepackage{todonotes}
\usepackage{hyperref}
\usepackage{multimedia}
\usepackage{graphicx}
\usepackage{indentfirst}
\usepackage{bm}
\usepackage{subfigure}
\usepackage{amsthm}
\usepackage{mathrsfs}
\usepackage{latexsym}
\usepackage{amsmath}
\usepackage{amssymb}
\usepackage{microtype}
\usepackage{relsize}
\usepackage{caption}
\usepackage{tikz}
\usepackage{hyperref}
\usepackage{dsfont}
\usepackage{amssymb}
\usepackage{calc}
\usepackage[marginal]{footmisc}
\makeatletter
\def\specialsection{\@startsection{section}{1}%
  \z@{\linespacing\@plus\linespacing}{.5\linespacing}%
  {\normalfont}}

\def\section{\@startsection{section}{1}%
  \z@{.7\linespacing\@plus\linespacing}{.5\linespacing}%
  {\normalfont\scshape}}
\makeatother

\newlength\mylength
\DeclareSymbolFont{SY}{U}{psy}{m}{n}
\DeclareMathSymbol{\emptyset}{\mathord}{SY}{'306}
\newtheorem{thm}{Theorem}[section]
\newtheorem{cor}[thm]{Corollary}
\newtheorem{lem}[thm]{Lemma}
\newtheorem{prop}[thm]{Proposition}
\theoremstyle{definition}
\newtheorem{defn}[thm]{Definition}

\newtheorem{rem}[thm]{Remark}
\newtheorem{example}[thm]{Example}
\newcommand{\dr}{\rightrightarrows}

\newcommand{\s}{\simeq}

\newcommand{\jp}{(J_{m_i},\partial J_{m_i})^{\Box (n-1)}}

\newcommand{\jt}{(J_{m_i}^{\Box n},\partial J_{m_i}^{\Box n}, \overline{J}_{m_i}^{\Box (n-1)})}

\newcommand{\pushoutcorner}[1][dr]{\save*!/#1-1.2pc/#1:(-1,1)@^{|-}\restore}
\newcommand{\Boxx}{\,\Box\,}
\newcommand{\qqed}{\hfill\Box}
\numberwithin{equation}{section}

\begin{document}
\captionsetup[figure]{labelfont={bf},labelformat={default},labelsep=space,name={Figure}}
\title[Calculating higher digraph homotopy groups]{Calculating higher digraph homotopy groups}

\author{Stephen Theriault} 
\address{Mathematical Sciences, University of Southampton, Southampton SO17 1BJ, United Kingdom} 
\email{S.D.Theriault@soton.ac.uk} 
\author{Jie Wu} 
\address{\vspace{-3.5mm}Beijing Key Laboratory of Topological Statistics and Applications for Complex Systems, Beijing }
\address{Institute of Mathematical Sciences and Applications, Beijing 101408, China.}
\email{wujie@bimsa.cn} 
\author{Shing-Tung Yau} 
\address{\vspace{-3.5mm}Beijing Institute of Mathematical Sciences and Applications, 101408, China}  
\address{Yau Mathematical Sciences Center, Tsinghua
University, Beijing 100084, China.} 
\email{styau@tsinghua.edu.cn} 
\author{Mengmeng Zhang$^{*}$}
\address{\vspace{-3.5mm}Beijing Key Laboratory of Topological Statistics and Applications for Complex Systems, Beijing }
\address{Institute of Mathematical Sciences and Applications, Beijing 101408, China.}
\email{mengmengzhang@bimsa.cn}

\thanks{$^{*}$ Research supported by start-up research funds of the second author and last author from Beijing Institute of Mathematical Sciences and Applications} 

\begin{abstract}
We give the first tractable and systematic examples of nontrivial higher digraph homotopy groups. To do this we define relative digraph homotopy groups and show these satisfy a long exact sequence analogous to the relative homotopy groups of spaces. We then define digraph suspension and Hurewicz homomorphisms and show they commute with each other. The existence of nontrivial digraph homotopy groups then reduces to the existence of corresponding groups in the degree $1$ path homology of digraphs. 
\end{abstract}
\maketitle

\section{Introduction}

Complex networks are frequently modelled by directed graphs (digraphs). Understanding the higher order structure of digraphs then gives tools for detecting and harnessing deeper structure within the networks. A framework for studying digraphs was developed by Grigor'yan, Lin, Muranov and Yau~\cite{GLMY12}. Inspired by ideas in algebraic topology, they constructed a homology theory for digraphs which was later enhanced by a corresponding homotopy theory for digraphs~\cite{GLMY15}. This marked the starting point for what is emerging as an entire new area referred to as GLMY theory, and which has been developed in numerous ways~\cite{Asao22,cubical,CK23,Chow3,DIMZ24,FI24,GWX22,GJM18, WL23}. 

In~\cite{GLMY15}, Grigor'yan, Lin, Muranov and Yau defined the fundamental group of a digraph $G$. They also defined a based loop digraph $LG$ that is the analogue of the based loop space in topology, and used this to define the $n^{th}$-higher digraph homotopy group $\pi_{n}(G)$ by taking the fundamental group of the $(n-1)$-fold iterated based loop digraph of $G$. 

A more direct approach to defining higher digraph homotopy groups was given in~\cite{LWYZ24}. This was based on the cubical approach to defining the higher homotopy groups of a space $X$ with basepoint~$x_{0}$, where if $I=[0,1]$ is the unit interval, $I^{n}$ is the product of $n$ copies of $I$ and $\partial I^{n}$ is the boundary of $I^{n}$, then $\pi_{n}(X)$ consists of homotopy classes of maps of pairs $(I^{n},\partial I^{n})\rightarrow (X,x_{0})$. There were several subtleties involved in making this definition consistent, stemming from the fact that there are many digraphs competing to be an analogue of~$I$, including $0\rightarrow 1$, $0\leftarrow 1$, and concatenations of these two digraphs. The outcome was higher digraph homotopy groups $\overline{\pi}_{n}(G)$ for $n\geq 1$, where $\overline{\pi}_{1}(G)$ is isomorphic to GLMY's $\pi_{1}(G)$, but the relationship between $\overline{\pi}_{n}(G)$ and $\pi_{n}(G)$ for $n\geq 2$ is unclear. The advantage of using $\overline{\pi}_{n}(G)$ was that it produced a long exact sequence of digraph homotopy groups associated to a based digraph map $A\rightarrow G$, analogous to the Puppe sequence of homotopy groups associated to a map of based spaces $X\rightarrow Y$. 

At this point two questions naturally arise. What other properties of higher homotopy groups for topological spaces have digraph analogues? Can any higher digraph homotopy groups be calculated? The purpose of this paper is to start addressing these questions. 

We address the first question by considering relative digraph homotopy groups. These are defined, in a manner analogous to how the definition of homotopy classes of maps $(I^{n},\partial I^{n})\rightarrow (X,x_{0})$ are extended to the relative case. This involves being careful to address the same subtleties that faced~\cite{LWYZ24} in defining the absolute digraph homotopy groups $\overline{\pi}_{n}(G)$. We show that, analogously to the relative homotopy groups of topological spaces, there is a canonical long exact sequence. 

\begin{thm}\label{introexact}
Let $(G,A)$ be a based digraph pair. Then there is a long exact sequence
$$\xymatrix@C=0.5cm{
 \cdots\ar[r]&\overline{\pi}_{n+2}(G,A)\ar[r]^-{\partial_{n+2}} &\overline{\pi}_{n+1}(A) \ar[r]^{i_{n+1}}&  \overline{\pi}_{n+1}(G) \ar[r]^-{j_{n+1}}&
    \overline{\pi}_{n+1}(G,A)\ar[r]^-{\partial_{n+1}}&
  \overline{\pi}_n(A) \ar[r]^{i_{n}} & \overline{\pi}_n(G)}$$
of based sets for any $n\geq0 $. If $n\geq 1$, it is a long exact sequence of groups.~$\qqed$
\end{thm}

Theorem~\ref{introexact} is then used to define a suspension homomorphism  
\[E_n:\overline{\pi}_{n}(G)\longrightarrow\overline{\pi}_{n+1}(\Sigma G)\] 
where $\Sigma G$ is the digraph suspension of $G$. The suspension homomorphism is similar to the suspension map for the homotopy groups of spaces. One fruitful way this is used topologically is in tandem with the Hurewicz homomorphism. So we next define a digraph analogue of the Hurewicz homomorphism. The ``cubical" nature of the definition of the digraph homotopy group $\overline{\pi}_{n}(G)$ means that the most natural homology theory in which to consider a Hurewicz map is the cubical homology theory of digraphs developed in~\cite{GMJ21}. In that same paper it was shown that there is a homomorphism $H^c_n(G)\rightarrow H_n(G)$ from the cubical digraph homology of $G$ to the GLMY path homology of $G$. The composite 
\[\widetilde{H}_n:\overline{\pi}_n(G)\rightarrow H^c_n(G)\rightarrow H_n(G)\] 
is referred to as the GLMY Hurewicz homomorphism. A suspension homomorphism $E'_n: H_n(G)\rightarrow H_{n+1}(\Sigma G)$ can also be defined in GLMY's path homology, and the key property of the Hurewicz homomorphism is that it commutes with suspension homomorphisms. 

\begin{thm}\label{intropcompa}
For any digraph $G$, the homotopy suspension and GLMY homology suspension are compatible, i.e. there is a commutative square $$\xymatrix{
     \overline{\pi}_{n}(G)\ar[r]^-{E_n}\ar[d]_{\widetilde{H}_n} & \overline{\pi}_{n+1}(\Sigma G)\ar[d]^{\widetilde{H}_{n+1}} \\
     H_{n}(G)\ar[r]^-{E'_n} & H_{n+1}(\Sigma G). }$$ 
\end{thm} 

It should be noted that the mechanics in carrying this out are much more complex than in the case of topological spaces, and stems from the delicate construction of $\overline{\pi}_n(G)$. One manifestation of this is that it is not known if the Hurewicz homomorphism is an isomorphism in degree $n$ if the digraph homotopy groups $\pi_{m}(G)$ are all zero for $m<n$. 

However, the existence of the commutative diagram in Theorem~\ref{intropcompa} can be used to great effect in addressing the question of whether any higher digraph homotopy groups can be calculated. Currently, only one example of a nontrivial digraph homotopy group is known. In~\cite[Example 4.22]{LWYZ24} it was shown there is a digraph $G$ such that $\overline{\pi}_2(G)\neq 0$, but what the actual group is remains unknown. We use Theorem~\ref{intropcompa} together with the isomorphism $\overline{\pi}_1(G)\cong\pi_1(G)$ and the Hurewicz isomorphism 
$\pi_1(G)_{ab}\rightarrow H_1(G)$~\cite{GLMY15} from the abelianization of $\pi_1(G)$ in order to prove the following. 

\begin{thm}\label{introexistence} 
For every $n\geq 1$ and $k\geq 1$ there are digraphs $G_n$, $G^k_n$, $G_n(2)$ and $G_n(3)$ such that: 
\begin{itemize} 
   \item $\overline{\pi}_n(G_n)$ has a $\mathbb{Z}$-summand; 
   \item $\overline{\pi}_n(G^k_n)$ has a $\mathbb{Z}^{\oplus k}$-summand; 
   \item $\overline{\pi}_n(G_n(2))$ has a $\mathbb{Z}/2\mathbb{Z}$-summand; 
   \item $\overline{\pi}_n(G_n(3))$ has a $\mathbb{Z}/3\mathbb{Z}$-summand. 
\end{itemize} 
\end{thm} 

In particular, there are nontrivial digraph homotopy groups for all $n$, they can have arbitrarily large rank, and they can have torsion. The restriction to $2$ and $3$ torsion is due to the remarkable fact that as yet no examples of digraphs are known for which $H_1(G)$ has a $\mathbb{Z}/m\mathbb{Z}$-summand for $m\geq 4$.

\section{Preliminaries}\label{sec:prelim}
In this section we review basic definitions and notations, based on the homology theories~\cite{GLMY12,GMJ21} and homotopy theory of digraphs
introduced by Grigor'yan, Lin, Muranov and Yau~\cite{GLMY15}.

\subsection{Basic definitions}
A \emph{digraph (directed graph)} $G=(V_{G}, A_{G})$ consists of a vertex set $V_{G}$ and an arrow
set $A_{G} \subset\{V_{G} \times V_{G} \backslash \operatorname{diag V_{G}}\}$ of ordered pairs of vertices,
where $\operatorname{diag V_{G}}$ is the diagonal set $\{(v,v)|v\in V_{G}\}$.
 Here, $(v, w) \in A_{G}$ is denoted by $v \rightarrow w$. For a fixed vertex $x_0$ in $G$, we say $G$ is a \emph{based digraph} with basepoint $x_0$. In what follows, we consider a digraph $G$ with basepoint~$x_0$.~\smallskip

Key examples of digraphs are line digraphs. For $n \geq 0$, let $I_n$ be the digraph whose
vertex set is $\{0,1,...,n\}$
and arrow set contains exactly one of the arrows $i \rightarrow(i+1)$ or $(i+1) \rightarrow i$
for any $i = 0, 1,...,n-1$. Such a digraph~$I_n$ is called a \emph{line digraph} of length $n$. A \emph{standard line digraph} $J_n$ is a line digraph of length $n$ with $i\rightarrow i+1$ if $i$ is even or 0, and $i\leftarrow i+1$ if $i$ is odd.
\smallskip

A \emph{digraph map} $f: G\rightarrow H$ is a map $f: V_G \rightarrow V_H $ such that for any
arrow $v\rightarrow w$ in $G$, either $f(v)\rightarrow f(w)$ or $f(v)= f(w)$.
\smallskip

A digraph $A$ is a \emph{subdigraph} of a digraph $G$, denoted by $A\subset G$, if $V(A)\subseteq V(G)$ and $E(A)\subseteq E(G)$. If $A$ is a subdigraph of a digraph $G$ then $(G,A)$ is a \emph{digraph pair}. A map of pairs
$f: (G,A)\rightarrow (H,B)$
is a digraph map
\(f: G\rightarrow H\)
such that
$f(A)\subset B$. More generally, if $C$ is also a subdigraph of~$A$ then $(G,A,C)$ is a
\emph{digraph triple}. A map of triples
$f: (G,A,C)\rightarrow (H, B,D)$
is a digraph map \(f: G\rightarrow H\)
such that $f(A)\subset B$ and $f(C)\subset D.$ In particular, if $C=A$ then $(G, A,A)$ is the
same as the digraph pair $(G,A)$.
\smallskip

If $G$ is a based digraph with basepoint $x_{0}$ then a \emph{digraph path} is a relative digraph map $\gamma: (I_n,n) \rightarrow (G,x_0)$. In the special case of a standard line digraph, a path
$\gamma: (J_n,n) \rightarrow (G,x_0)$ is called a \emph{standard path}.
\smallskip

Given a line digraph $I_{n}$, its \emph{boundary} $\partial I_{n}$ is the discrete subdigraph with vertices $\{0,n\}$ and no arrows, that is, $\partial I_{n}$ consists of the start and end vertices of $I_{n}$.
\smallskip

Let $G =(V_G, A_G)$ and $H=(V_H, A_H)$ be two digraphs.
 The \emph{box product} $G\Boxx H$ is the digraph whose vertex
set is $V_{G}\times V_{H}$ and whose arrow set consists of the arrows $(v,w)\rightarrow (v',w')$ in the cases when $v=v'$ and $w\rightarrow w'$, or $v\rightarrow v'$ and $w=w'$. The \emph{relative box product}
$(G,A)\Boxx (H,C)$ of two pairs is the pair $(G\Boxx H, A\Boxx H\cup G\Boxx B)$.
For pairs $(I_{m_{i}},\partial I_{m_{i}})$ for $1\leq i\leq n$, let $(I_{m_{i}},\partial I_{m_i})^{\Box n}$ be the relative box product  $$(I_{m_{1}},\partial I_{m_1})\Boxx (I_{m_{2}},\partial I_{m_2})\Boxx \cdots\Boxx (I_{m_{n}},\partial I_{m_n}), \quad\quad \quad n\geq 0.$$ If $n=0$ then, by convention, $(I_{m_{i}},\partial I_{m_i})^{\Box 0}$ is a discrete vertex. 
A digraph map
$f:(I_{m_{i}},\partial I_{m_i})^{\Box n}\rightarrow (G,A)$
is sometimes referred to as a \emph{grid digraph map}, and in the case of standard
line digraphs, a digraph map
$f:(J_{m_{i}},\partial J_{m_i})^{\Box n}\rightarrow (G,A)$
is sometimes referred to as a \emph{standard grid digraph map}.

\subsection{ Homotopies for Digraph Maps.} Two kinds of homotopy between digraph maps
were introduced in \cite{GLMY15}. The first homotopy is defined for arbitrary digraph maps while
the second one is defined only for paths. 


\begin{defn}Let $f,\text{ } g: (G, A)\rightarrow (H, B)$ be a map of digraph pairs. We
say that $f$ is \emph{homotopic} to $g$, denoted by $f\simeq g$, if there is a
line digraph~$I_n$ and a digraph map $F:(G,A)\,\Box\, I_n\rightarrow (H,B)$ such that $F|_{(G,A)\Box\{0\}} = f$,
$F|_{(G,A)\Box\{n\}} = g$, and $F|_{A\Box\{i\}}= f|_{A}=g|_{A}$ for all $0\leq i\leq n.$
If $n=1$, we say $f$ is \emph{direct homotopic} to $g$ and write $f\rightrightarrows g$.
\end{defn}

A digraph map $h:I_n\rightarrow I_m$ is called a \emph{shrinking map} if $h(0)=0$, $h(n)=m$ and
$h(i)\leq h(j)$ if $i\leq j$, that is, $h$ is a surjective digraph map preserving vertex order.

\begin{example}\label{shrink}Let $I_3 $ be $\xymatrix@C=0.5cm{ 0\ar[r]& 1& 2 \ar[r]\ar[l]&3}$ and $I_2$ be $\xymatrix@C=0.5cm{0\ar[r]& 1&2\ar[l]}.$ A shrinking map $h:I_3\rightarrow I_2$ is defined by $h(0) = 0$, $h(1)=1$ and $h(2)=h(3) = 2$.
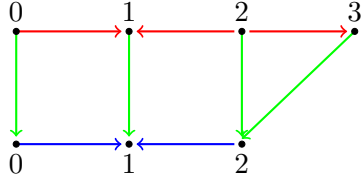
\begin{center}
\begin{tikzpicture}

  \draw[->, red, thick] (1,1.5) -- (2.4,1.5);
   \draw[<-, red, thick] (2.6,1.5) -- (3.9,1.5);
    \draw[->, red, thick] (4.1,1.5) -- (5.4,1.5);
\draw[->, blue, thick] (1,0) -- (2.4,0);
   \draw[<-, blue, thick] (2.6,0) -- (3.9,0);
   \draw[->, green, thick] (1,1.5) -- (1,0.1);
   \draw[->, green, thick] (2.5,1.5) -- (2.5,0.1);
   \draw[->, green, thick] (4,1.5) -- (4,0.1);
   \draw[->, green, thick] (5.5,1.5) -- (4,0.07);

  \node[above, black] (2) at (4,1.5) {2};
   \node[above, black] at (1,1.5) {0};
    \node[above, black] at (2.5,1.5) {1};
     \node[above, black] at (5.5,1.5) {3};
     \node[below, black] at (4,0) {2};
   \node[below, black] at (1,0) {0};
    \node[below, black] at (2.5,0) {1};

\filldraw (1,1.5) circle (.04)
(2.5,1.5) circle (.04)
(4 ,1.5) circle (.04)
(5.5,1.5) circle (.04)
(1,0) circle (.04)
(2.5 ,0) circle (.04)
(4,0) circle (.04);
\end{tikzpicture}
\end{center}
\end{example}
A relative digraph map $h:(I_{M_{i}},\partial I_{M_i})^{\Box n}\rightarrow (I_{m_{i}},\partial I_{m_i})^{\Box n}$ is called an \emph{$n$-dimensional shrinking map} if it is a box product of shrinking maps, that is, $h = h_1\Boxx h_2\Boxx \cdots \Boxx h_n $, where $h_i: (I_{M_i},\partial I_{M_i})\rightarrow (I_{m_i},\partial I_{m_i})$ is a shrinking map for $1 \leq i\leq n$. In particular, if $n=1$, a $1$-dimensional shrinking map is exactly a shrinking map defined in~\cite{GLMY15}. For convenience, we will omit the dimension of the shrinking map, simply calling it a shrinking map.

Let $f:(I_{m_{i}},\partial I_{m_i})^{\Box n}\rightarrow (G,x_0)$ be a relative digraph map. If $h:(I_{M_{i}},\partial I_{M_i})^{\Box n}\rightarrow (I_{m_{i}},\partial I_{m_i})^{\Box n}$ is a shrinking map then we call the composite $\overline{f}= f\circ h:(I_{M_{i}},\partial I_{M_i})^{\Box n}\rightarrow (G,x_0)$ a \emph{subdivision} of $f$.
\medskip

\begin{defn}~\cite{LWYZ24}
Let $f: (I_{m_{i}},\partial I_{m_i})^{\Box n}\rightarrow (G,x_0)$ and $g:(I_{n_{i}},\partial I_{n_i})^{\Box n}\rightarrow (G,x_0)$ be relative digraph maps. We call $f $ \emph{one-step $F$-homotopic} to $g$ if there exist subdivisions $\overline{f}$ and $\overline{g}$ of $f$ and $g$ respectively such that $\overline{f}\dr \overline{g}$, denoted by $f\simeq_{1}g$. In this case, we say $g$ is \emph{one-step inverse $F$-homotopic} to $f$ and write $g\s_{-1}f$.

 More generally, $f$ is \emph{$F$-homotopic} to $g$ if there is a finite sequence $\{f_i\}_{i=0}^l$ such that $f_{0}=f$, $f_{l}=g$ and there are one-step $F$-homotopies $f_i\s_{1}f_{i+1}$ or $f_i\s_{-1}f_{i+1}$ for $0\leq i<l$.
\end{defn}

By~~\cite[Remark 4.4]{LWYZ24}, any grid digraph map $f: (I_{m_{i}},\partial I_{m_i})^{\Box n}\rightarrow (G,x_0)$ 
is \mbox{$1$-step} $F$-homotopic to a standard grid digraph map $\overline{f}:(J_{M_i},\partial J_{M_i})^{\Box n}\rightarrow (G,x_0)$.  In what follows, this lets us always consider standard path and standard grid digraph maps.
\medskip 

\subsection{Digraph homotopy groups} 
First consider Hom sets of standard grid digraph maps. Let
$$Hom((J_{m_i},\partial J_{m_i})^{\Box n};(G,x_0))=\{f:(J_{m_i},\partial J_{m_i})^{\Box n}\rightarrow (G,x_0)\mid\mbox{$f$ is a digraph map}\}.$$
Let $(\mathbb{Z}^{\times n},\leq)$ be the poset where $(m_{1},\ldots,m_{n})\leq (s_{1},\ldots,s_{n})$ if and only if \mbox{$m_{i}\leq s_{i}$} for all $1\leq i\leq n$.
By~\cite{Mac,Maunder} there is a directed system $$\{Hom((J_{m_i},\partial J_{m_i})^{\Box n};(G,x_0)); l^{S}_{M} \}_{M\leq S}$$ where $$l^{S}_{M}:Hom((J_{m_i},\partial J_{m_i})^{\Box n};(G,x_0))\rightarrow Hom((J_{s_i},\partial J_{s_i})^{\Box n};(G,x_0))$$
sends $f$ to the map $\overline{f}$ defined by
$$\overline{f}(i_{1},\ldots,i_{n})=\left\{\begin{array}{ll}
    f(i_{1},\ldots,i_{n}), & \mbox{if $(i_{1},\ldots,i_{n})\leq M$} \\
    f(m_{1},\ldots,m_{n}), & \mbox{otherwise}. \end{array}\right.$$
The direct limit of $\{Hom((J_{m_i},\partial J_{m_i})^{\Box n};(G,x_0)); l_{M}^{S}\}_{M\leq S}$ is
$$\lim\limits_{\rightarrow}Hom((J_{m_i},\partial J_{m_i})^{\Box n};(G,x_0))= \bigsqcup\limits_{m_i,\text{ }\forall i}Hom((J_{m_i},\partial J_{m_i})^{\Box n};(G,x_0))/\sim ,$$ where $f_{M}\sim f_{S}$ if and only if there exists $V\in \mathds{Z}^{\times n}$ with $M\leq V$ and $S\leq V$ such that $l_{M}^{V}(f_{M}) = l_{S}^{V}(f_{S})$ for any $f_{M}\in Hom((J_{m_i},\partial J_{m_i})^{\Box n};(G,x_0))$ and $f_{S}\in Hom((J_{n_i},\partial J_{n_i})^{\Box n};(G,x_0))$.
To simplify notation, let
$$Hom((J,\partial J)^{\Box n};(G,x_0))=\lim\limits_{\rightarrow}Hom((J_{m_i},\partial J_{m_i})^{\Box n};(G,x_0)).$$ Let $$\overline{\pi}_n(G):=[(J,\partial J)^{\Box n};(G,x_0)]$$ be the set of $F$-homotopy classes of $Hom((J,\partial J)^{\Box n};(G,x_0))$. 

For each coordinate $j$ in the $n$-fold box product there is a multiplication $\mu^{j}$ defined on the set $ \overline{\pi}_n(G)$ by concatenation of standard grid maps. More precisely, let $f: (J_{m_i},\partial J_{m_i})^{\Box n} \rightarrow (G,x_0)$ and $g:(J_{n_i},\partial J_{n_i})^{\Box n}\rightarrow (G,x_0)$ be standard grid digraph maps and let $ M_i = max\{m_i,n_i\}$. The product of $f$ and~$g$ along the $j^{th}$-coordinate is the digraph map $$ \mu^j(f, g) :(J_{M_1}\Boxx \cdots \Boxx J_{m_j+n_j+1}\Boxx \cdots \Boxx J_{M_n}, \partial (J_{M_1}\Boxx \cdots \Boxx J_{m_j+n_j+1}  \Boxx \cdots \Boxx J_{M_n})) \rightarrow (G,x_0)$$ defined by
 $$  \mu^j(f, g)  (i_1,i_2,...,i_j,...,i_n) =\left\{
                     \begin{array}{ll}
                       f(i_1,i_2,...,i_j,...,i_n), & \hbox{$i_k \leq m_k$, $1\leq k\leq n$} \\
                       x_0, & \hbox{$m_j\leq i_j$ $\&$ $\exists$ $k \neq j$, $ m_k< i_k\leq M_k$} \\
                        g(i_1,i_2,...,i_j-m_j,...,i_n), & \hbox{$m_j< i_j $ $\&$ $\forall $ $k \neq j$, $1 \leq i_k\leq m_k $}\\
                         x_0 ,& \hbox{$n_j\leq i_j$ $\&$ $\exists$ $k \neq j$, $ n_k< i_k\leq M_k$.}
                     \end{array}
                   \right
. $$
The map $\mu^{j}$ induces a multiplication on $\overline{\pi}_{n}(G)$, which will also denoted by $\mu^{j}$.
The \emph{inverse} $f_j^{-1}$ of $f:(J_{m_i},\partial J_{m_i})^{\Box n} \rightarrow (G,x_0)$ along the $j^{th}$-coordinate is defined by  $$f_j^{-1}(i_1,i_2,...,i_j,...,i_n)\mapsto f(i_1,i_2,...,m_j-i_j,...,i_n).$$ 
By~\cite[Proposition 4.7]{LWYZ24} the multiplication $\mu^{j}$ defines a group structure on $\overline{\pi}_{n}(G)$ for any $n\geq 1$, and this multiplication is independent of the coordinate~$j$. 
\begin{defn}\cite{LWYZ24}\label{ndef}
Let $G$ be a based digraph. For $n\geq1$, the \emph{$n^{th}$-homotopy group} $\overline{\pi}_n(G)$ is defined by $$\overline{\pi}_n(G):= [(J,\partial J)^{\Box n};(G,x_0)].$$
\end{defn}
There is a different definition of the $n^{th}$-homotopy group $\pi_{n}(G)$ of a digraph~$G$ introduced by Grigor'yan, Lin, Muranov and Yau~\cite{GLMY15}. It is known that $\pi_1(G)\cong \overline{\pi}_1(G)$~\cite[Corollary 4.24]{LWYZ24}, but the relationship 
between $\pi_{n}(G)$ and $\overline{\pi}_{n}(G)$ for $n\geq 2$ is still unclear.

\subsection{Homology Theories for Digraphs}\label{cubical}
There are two homology theories for digraphs that are interrelated, both of which are homotopy invariants. One is the path homology of digraphs introduced by Grigor'yan, Lin, Muranov and Yau~\cite{GLMY12} in 2012, which is now referred to as GLMY homology. The other is cubical homology theory, introduced by  Grigor'yan, Muranov and Jimenez~\cite{GMJ21} in 2021. The two homology theories are connected by a homomorphism from cubical homology to path homology~\cite[Proposition~5.4]{GMJ21}. 

 Given a digraph $G$, there is an \emph{elementary path complex} $$(\Lambda_{\ast}(G),\widetilde{\partial}_{\ast})=(\{\Lambda_n(G)\}_{n\geq 0},\{\widetilde{\partial}_n\}_{n\geq0}),$$ where $\Lambda_n(G)$ is the free abelian group generated by all words $v_0v_1v_2\cdots v_{n} $ for \mbox{$v_i\in V(G)$}, and $\widetilde{\partial}_{n}$ is the boundary operator $$\widetilde{\partial}_n : \Lambda_n(G)\rightarrow\Lambda_{n-1}(G), \quad  av_0v_1v_2\cdots v_{n}\mapsto \sum\limits_{i=0}^{n}(-1)^id_i(av_0v_1v_2\cdots v_n),$$ with $d_{i}$ defined by $$d_i:\Lambda_n(G)\rightarrow\Lambda_{n-1}(G), \quad av_0v_1v_2\cdots v_{n}\mapsto av_0v_1v_2\cdots v_{i-1}v_{i+1}\cdots v_n.$$  The word $v_0v_1v_2\cdots v_{n} $ is called an \emph{elementary $n$-path}. If adjacent vertices in an elementary \mbox{$n$-path} are same the path is called a \emph{non-regular} $n$-path. The non-regular $n$-paths form a sub-chain complex $(D_{\ast}(G),\widetilde{\partial}_{\ast})=(\{D_n(G)\}_{n\geq0},\{\widetilde{\partial}_n\}_{n\geq 0})$ of $(\Lambda_{\ast}(G),\widetilde{\partial}_{\ast})$. There is a quotient chain complex $(\Lambda_{\ast}(G)/D_{\ast}(G),\partial_{\ast})$, called the \emph{regular chain complex}. It is easy to see that the regular chain complex is independent of the arrow set of the digraph $G$. 
 
 In considering the arrow set $A(G)$ of $G$, Grigor'yan, Lin, Muranov and Yau defined a graded group $A_{\ast}(G)=\{A_n(G)\}_{n\geq0}$.  Here, an $n$-path $v_0v_1v_2\cdots v_n$ is \emph{allowed} if each $v_iv_{i+1}$ is an arrow in $G$ for \mbox{$0\leq i\leq n-1$}, and $A_n(G)$ is the free abelian group generated by all allowed $n$-paths. Notably, $A_{\ast}(G)$ is not closed under the boundary operator $\widetilde{\partial}_{\ast}$ while 
  $A_{\ast}(G)$ can be seen as a graded subgroup of $\Lambda_{\ast}(G)/D_{\ast}(G)$. With this in mind, the \emph{$\partial$-invariant chain complex} $$(\Omega_{\ast}(G),\partial_{\ast})=(\{\Omega_{n}(G)\}_{n\geq 0}, \{\partial_{n}(G)\}_{n\geq0})$$ is defined by \begin{align}\label{1}
\Omega_{n}(G) &= \{A_n(G)\cap \partial^{-1}A_{n-1}(G)\}. \tag{1}
\end{align}
In what follows, $(\Omega_{\ast}(G),\partial_{\ast})$ is denoted more simply by $\Omega_{\ast}(G).$ The $n$-dimensional path homology group $H_n(G)$ of the digraph $G$ is defined by the $n^{th}$-homology group of the $\partial$-invariant chain complex, i.e.
$$H_n(G) = ker\, \partial_n / Im\, \partial_{n+1}.$$

The cubical homology of digraphs was introduced by Grigor'yan, Muranov and Jimenez in~\cite{GMJ21}. Denote the $n$-fold box product of line digraphs $J_1$ by $J^n_1$; this is called the standard unit $n$-dimensional grid. There are face maps$$F_{i,k}: J_1^{n-1}\rightarrow J_1^{n}; \quad (x_1,x_2,\cdots,x_{n-1})\mapsto (x_1,x_2,\cdots,x_{i-1},k,x_{i+1},\cdots,x_n),$$ where $k=0$, 1, and degeneracy maps $$S_{i}:J_1^{n+1}\rightarrow J_1^{n}; \quad (x_1,x_2,\cdots,x_{n+1})\mapsto (x_1,x_2,\cdots,x_{i-1},x_{i+1},\cdots,x_{n+1}).$$ A digraph map $f: J_1^{n}\rightarrow G$ is called an \emph{$n$-dimensional singular cubical chain}, and the $\mathds{R}$-module generated by all the $n$-dimensional singular cubical chains of $G$ is called the \textit{$n$-dimensional chain group}, denoted by $C_{n}^c(G)$. The \textit{cubical boundary operator} is defined by $$\widetilde{\partial}^{c}_n:C_n^c(G)\rightarrow C_{n-1}^c(G); \quad f \mapsto \sum\limits_{i=1}^n(-1)^i d_i^c f,$$ where $$d_i^c f = f\circ F_{i,0}-f\circ F_{i,1}.$$
Moreover, if there exists a degeneracy map $S_i$ and a digraph map $g$ such that $f= g\circ S_i$, then we call $f$ degenerate. Let $D_n^c(G)$ be the free $\mathds{R}$-module generated by all degenerate $n$-dimensional singular cubical chains of $G$. Similar to the path homology group, there is a quotient cubical chain complex $(C_{\ast}^c(G)/D_{\ast}^c(G),  \partial_{\ast}^c)$, and its $n$-dimensional homology group is called the \textit{$n$-dimensional cubical homology group}, denoted by $H_n^c(G).$

To describe the relation between the two homology theories, some notation is needed. Let $P(a, b)$ be the set of all allowed paths of length $n$ from the vertex $(0,0,\cdots,0)$ to the vertex $(1,1,\cdots,1)$ in the digraph $J_1^n$. If $\alpha=\{\alpha_{k}\}_{k=0}^{m}$ is such a path, that is, if there is a sequence of arrows $\alpha_{0}\rightarrow\alpha_{1}\rightarrow\cdots\rightarrow\alpha_{m}$, observe that for each $1\leq k\leq m$, the vertex $\alpha_{k}$ is obtained from the vertex $\alpha_{k-1}$ by changing precisely one $0$ to a $1$. Suppose this change occurs in position~$i_{k}$. Let $\sigma(\alpha)$ be the inversion number of the sequence $\{i_{1},\ldots,i_{m}\}$.  Consider the $n$-path
$$
\omega_{n}=\sum_{\alpha \in P(a, b)}(-1)^{\sigma(\alpha)} \alpha
.$$
In~\cite{GMY14}, Grigor'yan, Muranov and Yau showed that $\omega_{n}$ generates $\Omega_n\left(J_{1}^n\right)$. They then constructed a homomorphism $$\iota_n: \Omega_n^c(G)\rightarrow \Omega_n(G), \quad f\mapsto f_n(w_n),$$ where $f_n: \Omega_{n}(J_1^n)\rightarrow  \Omega_{n}(G)$ is the chain map induced by digraph map $f$, and showed that $\iota_n$ is a chain map, implying that it induces a homomorphism 
\begin{equation}\label{Lndef} 
L_n: H_n^c(G)\rightarrow H_n(G). 
\end{equation} 

\section{Relative Homotopy Groups of Digraph Pairs}
The purpose of this section is to  define the relative homotopy groups of digraphs, which generalize the homotopy groups of digraphs in~\cite{LWYZ24}. We prove some initial properties and show that there is an exact sequence of relative homotopy groups.

 \begin{defn}Let $(G,A,C)$ and $(H,B,D)$ be digraph triples. The \emph{triple box product} is given by $$(G,A,C)\Box (H,B,D)=(G\Boxx H, A\Boxx H \cup G\Boxx B, G\Boxx D \cup C\Boxx H).$$
\end{defn}
Recall that the digraph triple $(G,A,A)$ is the digraph pair $(G,A)$. In this case, the definition of triple box product implies that $(G, A, A) \Boxx (H,B,B)$  is exactly the relative box product $(G,A)\Boxx(H,B)$. For simplicity, we will write $(G,A,A)$ as $(G,A)$.

To define a relative homotopy group of a digraph, we will need an $n$-dimensional grid triple $(J_{m_i}^{\Box n},\partial J_{m_i}^{\Box n},$ $ \overline{J}_{m_i}^{\Box (n-1)})$. Let us start with an example.

\begin{example} Let$(J_2,\partial J_2,2)$ and $(J_3,\partial J_3)=(J_3,\partial J_3,\partial J_3)$  be digraph triples, where $2$ is the digraph consisting of one vertex labelled $2$. Then the triple box product of
$(J_2,\partial J_2,2)$ and  $(J_3,\partial J_3)$ is
$$(J_2,\partial J_2,2) \Boxx(J_3,\partial J_3)= (J_2\Boxx J_3,\partial J_2\Boxx J_3 \cup J_2\Boxx \partial J_3, 2 \Boxx J_3\cup J_2\Boxx \partial J_3 ),$$
which is shown as follows:
$$
\xymatrix@R=3.5em{
(0,0) \ar@[red][d] \ar@[green][r]  &(1,0)\ar[d] & (2,0)  \ar@[green][l] \ar@[green][r]  \ar[d]& (3,0)\ar@[red][d] \\
(0,1) \ar[r]        & (1,1)        & (2,1)  \ar[l]  \ar[r]       & (3,1)      \\
(0,2) \ar@[red][r]\ar@[red][u]  & (1,2)\ar[u]  & (2,2)  \ar@[red][l]\ar[u]\ar@[red][r]   & (3,2). \ar@[red][u]
 }
 $$
Here, the subdigraph given by the red and green arrows is $\partial J_2\Boxx J_3 \cup J_2\Boxx \partial J_3$ and the subdigraph given by the red arrows is $J_3 \Boxx 2\cup J_2\Boxx \partial J_3$. Similarly, the triple box product $(J_{m_1}, \partial J_{m_1},m_1) \Boxx (J_{m_2}, \partial J_{m_2})$ of any digraph triples $(J_{m_1}, \partial J_{m_1}, m_1)$ and $(J_{m_2}, \partial J_{m_2})$ is a digraph triple given by a 2-dimensional grid, its boundary, and its boundary with the ``top" removed.
\end{example}

More generally, consider the $n$-fold triple box product of standard line digraphs,
\begin{multline*} (J_{m_1},\partial J_{m_1},m_1) \Boxx(J_{m_i},\partial J_{m_i})^{\Box (n-1)}= \\[2mm]
   (J_{m_i}^{\Box n},\partial J_{m_1}\Boxx J_{m_i}^{\Box (n-1)}\cup \partial J_{m_1} \Boxx J_{m_i}^{\Box (n-1)}, m_1\Boxx J_{m_i}^{\Box (n-1)} \cup  J_{m_1}\Boxx\partial J_{m_i}^{\Box (n-1)}).
 \end{multline*}
 Observe that
 $$\partial J_{m_{i}}^{\Box n}=\partial J_{m_1}\Boxx J_{m_i}^{\Box (n-1)}\cup \partial J_{m_1} \Boxx J_{m_i}^{\Box (n-1)}.$$
 Define $\overline{J}_{m_i}^{\Box (n-1)}$ by
 $$\overline{J}_{m_i}^{\Box (n-1)}=m_1\Boxx J_{m_i}^{\Box (n-1)} \cup  J_{m_1}\Boxx\partial J_{m_i}^{\Box (n-1)}.$$
Then $$(J_{m_i}^{\Box n},\partial J_{m_i}^{\Box n}, \overline{J}_{m_i}^{\Box (n-1)})=(J_{m_1},\partial J_{m_1},m_1)\Boxx (J_{m_i},\partial J_{m_i})^{\Box (n-1)}.$$  A digraph map $$f: (J_{m_i}^{\Box n},\partial J_{m_i}^{\Box n}, \overline{J}_{m_i}^{\Box (n-1)})\rightarrow (G,x_0)$$ will be called a \emph{triple grid map}. In the special case when $n=0$, we refer to $(J_{m_i},\partial J_{m_i})^{\Box 0}$ as the discrete vertex $0.$

Next, we consider a homotopy between triple grid maps by using subdivision. Observe that if $h: I_m\rightarrow I_n$ is a shrinking map then $h$ sends $\partial I_{m}$ to $\partial I_{n}$ and $m$ to $n$, so $h$ can be regarded as a map of digraph triples $(I_m,\partial I_m,m)\rightarrow (I_n,\partial I_n,n)$. Similarly, an $n$-dimensional shrinking map
$h:(I_{M},\partial I_{M})^{\Box n}\rightarrow (I_{m},\partial I_{m})^{\Box n}$ is a map of digraph triples
$(I_{M_{i}}^{\Box n},\partial I_{M_{i}}^{\Box n},\overline{I}_{M_{i}}^{\Box (n-1)})\rightarrow (I_{m_{i}}^{\Box n},\partial I_{m_{i}}^{\Box n},\overline{I}_{m_{i}}^{\Box (n-1)})$.
\begin{defn}
 Let $f: \jt\rightarrow (G,A,x_0)$ and $g:(J_{l_{i}}^{\Box n},\partial J_{l_i}^{\Box n}, $ $\overline{J}_{l_i}^{\Box(n-1)})\rightarrow (G,A,x_0)$ be triple grid maps. If there exist subdivisions $\overline{f}$ and $\overline{g}$ of $f$ and $g$ respectively such that $\overline{f}\dr\overline{g}$, then $f $ is said to be \emph{one-step $F$-homotopic rel $A$} to $g$, denoted by $f\simeq_{1}g\, (rel\ A)$.  Also, we call $g$ \emph{one-step inverse $F$-homotopic rel $A$} to $f$ and write $g\s_{-1}f (rel\ A)$.

 More generally, we say that $f$ is \emph{$F$-homotopic} to $g$, and write $f\simeq_{F} g$, if there is a finite sequence~$\{f_i\}_{i=0}^l$ such that $f_{0}=f$, $f_{l}=g$ and there are one-step $F$-homotopies rel $A$ $f_i\s_{1}f_{i+1}$ or $f_i\s_{-1}f_{i+1}$ for $0\leq i<l$.
\end{defn}

We now give a digraph triple analogue of the directed system of Hom sets discussed in Section~\ref{sec:prelim}. Let
\begin{multline*}
  Hom((J_{m_{i}}^{\Box n},\partial J_{m_i}^{\Box n},\overline{J}_{m_i}^{\Box(n-1)}); (G,A,x_0))= \\
   \{f\colon (J_{m_{i}}^{\Box n},\partial J_{m_i}^{\Box n},\overline{J}_{m_i}^{\Box(n-1)})\rightarrow (G,A,x_{0})\mid \mbox{$f$ is a digraph map}\}.
 \end{multline*}
Using the poset  $(\mathds{Z}^{\times n},\leq)$, there is a directed system
$$\{Hom(\jt;(G,A,x_0)); L_{M}^{S}\}_{M\leq S}$$
where $$L_{M}^{S}:Hom(\jt;(G,A,x_0)) \rightarrow Hom((J_{s_{i}}^{\Box n},\partial J_{s_i}^{\Box n},  \overline{J}_{s_i}^{\Box(n-1)}); (G,A,x_0))$$
sends $f$ to the map $\overline{f}$ defined by
$$\overline{f}(i_{1},\ldots,i_{n})=\left\{\begin{array}{ll}
    f(i_{1},\ldots,i_{n}) & \mbox{if $(i_{1},\ldots,i_{n})\leq M$} \\
    f(m_{1},\ldots,m_{n}) & \mbox{otherwise}. \end{array}\right.$$
The direct limit of $\{Hom(\jt;(G,A,x_0)); L_{M}^{S}\}_{M\leq S}$ is
  $$ \lim\limits_{\rightarrow}Hom(\jt;(G,A,x_0))=
\bigsqcup\limits_{m_i,\forall i} Hom(\jt;(G,A,x_0))/\sim
$$
where $f_{M}\sim f_{V}$ if and only if there exists $V\in \mathds{Z}^{\times n}$ with $M\leq V$ and $S\leq V$ such that $L_{M}^{V}(f_{M}) = L_{S}^{V}(f(S))$ for any $$f_{M}\in Hom((J_{m_i}^{\Box n},\partial J_{m_i}^{\Box n},\overline{J}_{m_i}^{\Box (n-1)});(G,A,x_0))\quad\mbox{and}\quad f_{S}\in Hom((J_{s_i}^{\Box n},\partial J_{s_i}^{\Box n},\overline{J}_{s_i}^{\Box (n-1)});(G,A,x_0)).$$ To simplify notation, let
$$Hom((J,\partial J,\overline{J}^{\Box n}),(G,A,x_0))=\lim\limits_{\rightarrow}Hom(\jt;(G,A,x_0)).$$
Write the equivalence class of $f:\jt \rightarrow (G,A,x_0)$ as $\{f\}$.

For elements $\{f\},\text{ }\{g\}\in \lim\limits_{\rightarrow}Hom(\jt;(G,A,x_0))$,
write $\{f\}\s_{1}\{g\}$ if and only if
$f\s_1 g$. Thus we do not distinguish between $\{f\}$ and $f$. Let
$$[(J,\partial J, \overline{J})^{\Box n};(G,A,x_0)]$$
be the set of $F$-homotopy classes of $Hom((J,\partial J, \overline{J})^{\Box n};(G,A,x_0))$
and write $[f]$ for the $F$-homotopy class of $f$. This will be the underlying set for the $n^{th}$-relative homotopy group $\overline{\pi}_n(G,A)$ of the digraph pair $(G,A)$.

\begin{thm}\label{relgroup} 
Let $G$ be a based digraph with basepoint $x_0$ and let $A$ be a subdigraph of $G$. If $n\geq 2$ then the set $[(J^{\Box n},\partial J^{\Box n},\overline{J}^{\Box (n-1)});(G, A, x_0)]$ forms a group with concatenation as multiplication.

\begin{proof}
Suppose there are digraph maps $$f:(J_{m_{i}}^{\Box n},\partial J_{m_{i}}^{\Box n}, \overline{J}_{m_i}^{\Box (n-1)})\rightarrow (G,A,x_0)\quad\mbox{and}\quad g: (J_{l_{i}}^{\Box n},\partial J_{l_{i}}^{\Box n}, \overline{J}_{l_i}^{\Box (n-1)})\rightarrow (G,A,x_0).$$ A multiplication can be defined along each coordinate of the $n$-fold box product, so for $2\leq j\leq n$ we will define a family of multiplications $\{\mu^j\}_{j=2}^n$
$$
  \mu^j:Hom((J,\partial J,\overline{J})^{\Boxx n};(G,A,x_0))\times Hom((J,\partial J,\overline{J})^{\Boxx n};(G,A,x_0))\rightarrow Hom((J,\partial J,\overline{J})^{\Boxx n};(G,A,x_0)).
$$
In box coordinate $j$, the input comes from the line digraphs $J_{m_{j}}$ for $f$ and $J_{l_{j}}$ for $g$. The lengths $m_{j}$ and $l_{j}$ may be different, so in order for the domain of the product $\mu^j(f, g)$ of $f$ and $g$ to be a grid we first need to extend the domains of~$f$ and $g$ to be grids. Note the extended grids will be of even length as the standard line digraphs $J_{\ell}$ are all of even length. Note also that such a subdivision does not change the homotopy type and we will show that the multiplication is independent of the representative of the homotopy class. Hence, we always assume that $m_i$ and $l_i$ are even and non-zero for $i\geq 1$. 

Let $M_i = max\{l_i,m_i\}$.
Define $\widetilde{f}_j$ and $\widetilde{g}_j$ as in ~\cite{LWYZ24}: 
$$\widetilde{f}_j:(J_{M_i}^{\Box(j-1)},\partial J_{M_i}^{\Box(j-1)},\overline{J}_{M_i}^{\Box(j-1)})\Boxx(J_{m_{j}},\partial J_{m_{j}}) \Boxx (J_{M_k}^{\Box(n-j)}, \partial J_{M_k}^{\Box(n-j)})\rightarrow (G,A,x_0) $$ is defined by \begin{align*}
    \widetilde{f}_j(i_1,i_2,...,i_j,...,i_n) &=
    \begin{cases}
       f(i_1,i_2,...,i_j,...,i_n), & \text{if } i_j \leq m_j, \\
       x_0, & \text{if } i_j > m_j
    \end{cases}
\end{align*}
and $$\widetilde{g}_j: (J_{M_i}^{\Box(j-1)},\partial J_{M_i}^{\Box(j-1)},\overline{J}_{M_i}^{\Box(j-1)})\Boxx(J_{l_{j}},\partial J_{l_{j}}) \Boxx (J_{M_k}^{\Box(n-j)}, \partial J_{M_k}^{\Box(n-j)})\rightarrow (G,A,x_0)  $$ is defined by \begin{align*}
\widetilde{g}_j(i_1,i_2,...,i_j,...,i_n) &=
    \begin{cases}
       g(i_1,i_2,...,i_j,...,i_n), & \text{if } i_j \leq l_j, \\
       x_0, & \text{if } i_j >l_j.
    \end{cases}
\end{align*}

Using $\widetilde{f}_{j}$ and $\widetilde{g}_{j}$, for $2\leq j\leq n$ define $$ \mu^j(f,g):(J_{M_i}^{\Box(j-1)},\partial J_{M_i}^{\Box(j-1)},\overline{J}_{M_i}^{\Box(j-1)})\Boxx(J_{m_{j}+l_j+1},\partial J_{m_{j}+l_j+1}) \Boxx (J_{M_k}^{\Box(n-j)}, \partial J_{M_k}^{\Box(n-j)})\rightarrow (G,A,x_0) $$
by
\begin{align*}
    \mu^j(f,g)(i_1,i_2,...,i_j,...,i_n) =
    \begin{cases}
        \widetilde{f}(i_1,i_2,...,i_j,...,i_n), & \text{if } i_j \leq m_j, \\
        \widetilde{g}(i_1,i_2,...,i_j-m_j,...,i_n), & \text{if } i_j > m_j.
    \end{cases}
\end{align*}
That is, $\mu^j(f,g) = \widetilde{f}_j\vee_j \widetilde{g}_j,$ where $\vee_j$ is concatenation along the $j^{th}$ coordinate.

Next, we check that $\mu^j$ induces a map on homotopy sets: 
\begin{multline*}
[(J^{\Box n},\partial J^{\Box n},\overline{J}^{\Box (n-1)});(G,A,x_0)]\times[(J^{\Box n},\partial J^{\Box n},\overline{J}^{\Box (n-1)});(G,A,x_0)]\rightarrow \\
[(J^{\Box n},\partial J^{\Box n},\overline{J}^{\Box (n-1)});(G,A,x_0)].
\end{multline*}
Suppose that $f\s_1 f^{'}$. Then there exist subdivisions $\overline{f}$ and $ \overline{f^{'}}$ of $f$ and $f^{'}$ respectively
such that $\overline{f}\rightrightarrows \overline{f^{'}}$. By~\cite[Corollary 4.6]{LWYZ24},
there is a common subdivision $\overline{\widetilde{f}}$ of $\widetilde{f}$ and $\overline{f}$ and a
common subdivision $ \overline{\widetilde{f^{'}}}$ of $\widetilde{f^{'}}$ and $\overline{f^{'}}$ such that $\overline{\widetilde{f}}\dr \overline{\widetilde{f^{'}}}$.
This implies that $\mu^j(\overline{\widetilde{f}},g)\dr \mu^j(\overline{\widetilde{f^{'}}},g)$. Therefore  $\mu^j(\overline{\widetilde{f}},g)$ and $\mu^j(\overline{\widetilde{f'}},g)$ are subdivisions of $\mu^j(f,g)$ and $\mu^j(f',g)$ respectively, so $\mu^j(f,g)\s_1 \mu^j(f^{'},g)$. Similarly, if $g\s_1 g^{'}$, then $\mu^j(f,g)\s_1 \mu^j(f,g^{'})$. Moreover, $\mu^j(f,g)\s_1 \mu^j(f,g^{'})\s_1 \mu^j(f^{'},g^{'})$. Hence  if $f\s_1 f^{'}$ or $g\s_1 g^{'}$, then $\mu^j(f,g)\s_1 \mu^j(f^{'},g^{'})$.
Consequently, $\mu^j$ induces a well-defined map on homotopy sets.

The unit element $e_{x_0}$ is easily checked to be the constant digraph map $J_2^{\Box n}\rightarrow x_0$. Associativity and the existence of inverses follow as in the proof of~\cite[Proposition 4.7]{LWYZ24}. 
\end{proof}
\end{thm} 
\begin{rem}\label{mujindependence} 
The multiplication constructed in the proof of Theorem~\ref{relgroup} was based on concatenation in box coordinate $j$, with $j\geq 2$. But as in the proof of~\cite[Proposition 4.7]{LWYZ24}, any choice of box coordinate $j\geq 2$ induces the same multiplication on 
$[(J,\partial J,\overline{J} )^{\Box n}; (G, A, x_0)]$. 
\end{rem}

In what follows the multiplication on
$[(J,\partial J,\overline{J})^{\Box n};(G, A, x_0)]$
will be taken to be in the second coordinate and will be denoted by ``$\cdot$".
\begin{defn}
Let $(G,A)$ be a digraph pair with basepoint $x_0$. For $n\geq 2$, the \emph{$n^{th}$-relative homotopy group} of $(G, A)$ is $$\overline{\pi}_n(G, A) = [(J,\partial J,\overline{J})^{\Box n};(G, A, x_0)].$$
\end{defn}
When $n=1$, by definition, $\overline{\pi}_1(G, A) = [(J,\partial J, \overline{J});(G, A, x_0)]. $ However, when attempting to multiply, the starting vertex of $g$ may be not same as the end vertex of $f$, so the concatenation of digraph maps $f$ and $g$ does not make sense. Hence, $\overline{\pi}_1(G, A)$ is a set, instead of group. 

We now establish several fundamental properties of relative digraph homotopy groups.

\begin{prop}\label{inducedhom} 
Any based digraph map $l\colon (G,A,x_0)\rightarrow (H,B,y_0)$ induces a morphism
$l_{n}\colon\overline{\pi}_n(G,A)\rightarrow \overline{\pi}_n(H,B)$ for $n\geq 1$. If $n\geq 2$ then
$l_{n}$ is a homomorphism.
\end{prop}

\begin{proof}
Any based digraph map $l: (G,A,x_0)\rightarrow (H,B,y_0)$ induces a morphism of directed systems of Hom sets and therefore a map of their inverse limits, that is, $l$ induces a map
$$Hom((J,\partial J,\overline{J})^{\Box n};(G,A,x_0))\rightarrow Hom((J,\partial J,\overline{J})^{\Box n};(G,A,x_0)).$$
This then passes to $F$-homotopy classes, giving a map
$$[(J,\partial J,\overline{J})^{\Box n};(G,A,x_0)]\rightarrow [(J,\partial J,\overline{J})^{\Box n};(G,A,x_0)],$$
that is, a map
$$l_{n}:\overline{\pi}_n(G,A)\rightarrow \overline{\pi}_n(H,B)$$ for any $n\geq1.$ 

Next, for any digraph grid maps $f\colon(J_{m_{i}}^{\Box n},\partial J_{m_{i}}^{\Box n}, \overline{J}_{m_i}^{\Box (n-1)})\rightarrow (G,A,x_0)$ and $g\colon (J_{l_{i}}^{\Box n},\partial J_{l_{i}}^{\Box n}, \overline{J}_{l_i}^{\Box (n-1)})$
$\rightarrow (G,A,x_0)$, consider the string of equalities 
$$l_n(f \cdot g) = l \circ (f \cdot g ) = (l \circ f) \cdot (l \circ g)  = l_n(f )\cdot l_n(g).$$ The first and third hold by definition of $l_{n}$ while the second holds since $\cdot$ is concatenation. Thus, $l_n$ is a homomorphism for $n\geq 2$. 
\end{proof}

\begin{prop}\label{functoriality} For any integer $n\geq 2$, $\overline{\pi}_n: \mathcal{D}ig_{\ast} \rightarrow \mathcal{G}p$ is a functor from the category of based digraph pairs to the category of groups.
\begin{proof}
Arguing as for the homomorphism property in the proof of Proposition~\ref{inducedhom} shows that  $\overline{\pi}_n(\phi \circ \psi) = \overline{\pi}_n(\phi)\circ \overline{\pi}_n(\psi)$ and $\overline{\pi}_n (id_G) = id_{\overline{\pi}_n(G)}$ for all $n\geq 1$.
\end{proof}
\end{prop}

\begin{prop}\label{map}
If $f\simeq g: (G,A,x_0)\rightarrow (H,B,y_0)$, then $\overline{\pi}_n(f)= \overline{\pi}_n(g)$ for $n\geq1$.
\end{prop}

\begin{proof}
The argument in~\cite[Proposition 4.1]{LWYZ24} for the absolute case is easily adapted.
\end{proof}

\begin{cor}
If $(G, A, x_0)\s (H , B,y_0)$, then $\overline{\pi}_n(G,A)\cong \overline{\pi}_n(H,B)$ for $n\geq 1.$ 
\end{cor} 
\vspace{-0.8cm}~$\qqed$\medskip 

\begin{prop} \label{gp} For any based digraph pair $(G,A)$,
the group $\overline{\pi}_n(G,A)$ is abelian if $n\geq 3$.
\begin{proof}
Suppose there are $n$-dimensional grid maps $f: (J_{m_1},\partial J_{m_1}, m_1)\Boxx (J_{m_i},\partial J_{m_i})^{\Box (n-1)}  \rightarrow (G,A,x_0)$ and $g: (J_{l_1},\partial J_{l_1}, l_1) \Boxx (J_{l_i},\partial J_{l_i})^{\Box (n-1)} 
 \rightarrow (G,A,x_0)$.  
By Remark~\ref{mujindependence} and the definition of the concatenation multiplication, we have
$[\mu^3(f,0)]=[\mu^{2}(f,0)]=[f]\cdot[0] =[f]$ and $[\mu^3(0,f)]=[\mu^{2}(0,f)]=[0]\cdot[f] =[f]$.  
Therefore $$[\mu^3(f,0)]\cdot[\mu^3(0,g)] = [f]\cdot[g] =[\mu^2(0,f)]\cdot[\mu^2(g,0)] =[\mu^3(0,g)]\cdot[\mu^3(f,0)] = [g]\cdot [f] .
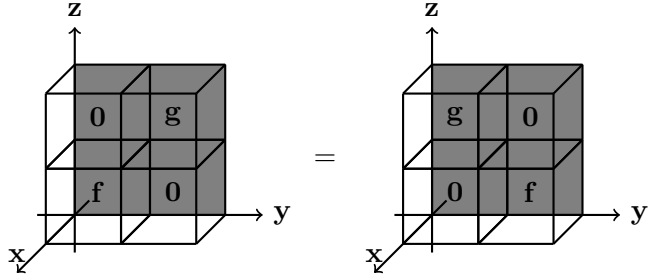
$$
\begin{center}
\begin{tikzpicture}
    \def\l{1}
    
\foreach \x in {0,1} {
        \foreach \y in {0,1} {
            \fill[gray, opacity=0.5] (\x, \y, 0) -- (\x+1, \y, 0) -- (\x+1, \y+1, 0) -- (\x, \y+1, 0) -- cycle;
        }
    }
    \foreach \x in {0,1} {
        \foreach \y in {0,1} {
            \draw[thick] (\x*\l,\y*\l,0) -- (\x*\l+\l,\y*\l,0);
            \draw[thick] (\x*\l,\y*\l,0) -- (\x*\l,\y*\l+\l,0);
            \draw[thick] (\x*\l+\l,\y*\l,0) -- (\x*\l+\l,\y*\l+\l,0);
            \draw[thick] (\x*\l,\y*\l+\l,0) -- (\x*\l+\l,\y*\l+\l,0);

            \draw[thick] (\x*\l,\y*\l, \l) -- (\x*\l+\l,\y*\l, \l);
            \draw[thick] (\x*\l,\y*\l, \l) -- (\x*\l,\y*\l+\l, \l);
            \draw[thick] (\x*\l+\l,\y*\l, \l) -- (\x*\l+\l,\y*\l+\l, \l);
            \draw[thick] (\x*\l,\y*\l+\l, \l) -- (\x*\l+\l,\y*\l+\l, \l);

            \draw[thick] (\x*\l,\y*\l,0) -- (\x*\l,\y*\l, \l);
            \draw[thick] (\x*\l+\l,\y*\l,0) -- (\x*\l+\l,\y*\l, \l);
            \draw[thick] (\x*\l,\y*\l+\l,0) -- (\x*\l,\y*\l+\l, \l);
            \draw[thick] (\x*\l+\l,\y*\l+\l,0) -- (\x*\l+\l,\y*\l+\l, \l);
        }
    }

    \node at (0.5, 0.5, 0.5) {\textbf{f}};
    \node at (1.5, 0.5, 0.5) {\textbf{0}};
    \node at (0.5, 1.5, 0.5) {\textbf{0}};
    \node at (1.5, 1.5, 0.5) {\textbf{g}};

\draw[thick] (3.5,1.2,1.2) node[right] {=};
    \draw[->, thick] (-0.5,0,0) -- (2.5,0,0) node[right] {\textbf{y}};
    \draw[->, thick] (0,-0.5,0) -- (0,2.5,0) node[above] {\textbf{z}};
    \draw[->, thick] (0,0,-0.5) -- (0,0,2) node[above] {\textbf{x}};
\end{tikzpicture}
\begin{tikzpicture}
    \def\l{1}

    \foreach \x in {0,1} {
        \foreach \y in {0,1} {
            \fill[gray, opacity=0.5] (\x, \y, 0) -- (\x+1, \y, 0) -- (\x+1, \y+1, 0) -- (\x, \y+1, 0) -- cycle;
        }
    }

    \foreach \x in {0,1} {
        \foreach \y in {0,1} {
            \draw[thick] (\x*\l,\y*\l,0) -- (\x*\l+\l,\y*\l,0);
            \draw[thick] (\x*\l,\y*\l,0) -- (\x*\l,\y*\l+\l,0);
            \draw[thick] (\x*\l+\l,\y*\l,0) -- (\x*\l+\l,\y*\l+\l,0);
            \draw[thick] (\x*\l,\y*\l+\l,0) -- (\x*\l+\l,\y*\l+\l,0);

            \draw[thick] (\x*\l,\y*\l, \l) -- (\x*\l+\l,\y*\l, \l);
            \draw[thick] (\x*\l,\y*\l, \l) -- (\x*\l,\y*\l+\l, \l);
            \draw[thick] (\x*\l+\l,\y*\l, \l) -- (\x*\l+\l,\y*\l+\l, \l);
            \draw[thick] (\x*\l,\y*\l+\l, \l) -- (\x*\l+\l,\y*\l+\l, \l);

            \draw[thick] (\x*\l,\y*\l,0) -- (\x*\l,\y*\l, \l);
            \draw[thick] (\x*\l+\l,\y*\l,0) -- (\x*\l+\l,\y*\l, \l);
            \draw[thick] (\x*\l,\y*\l+\l,0) -- (\x*\l,\y*\l+\l, \l);
            \draw[thick] (\x*\l+\l,\y*\l+\l,0) -- (\x*\l+\l,\y*\l+\l, \l);
        }
    }

    \node at (0.5, 0.5, 0.5) {\textbf{0}};
    \node at (1.5, 0.5, 0.5) {\textbf{f}};
    \node at (0.5, 1.5, 0.5) {\textbf{g}};
    \node at (1.5, 1.5, 0.5) {\textbf{0}};

    \draw[->, thick] (-0.5,0,0) -- (2.5,0,0) node[right] {\textbf{y}};
    \draw[->, thick] (0,-0.5,0) -- (0,2.5,0) node[above] {\textbf{z}};
    \draw[->, thick] (0,0,-0.5) -- (0,0,2) node[above] {\textbf{x}};

\end{tikzpicture}
\end{center}
Hence, $[f]\cdot[g] = [g]\cdot [f].$ Thus $\overline{\pi}_n(G,A)$ is abelian for $n\geq 3.$
\end{proof}
\end{prop}
Similar to topological spaces, if $n=1$ then $\overline{\pi}_1(G,A)$ is the set consisting of all the homotopy classes of paths $\gamma: (J_{m},\partial J_m, m)\rightarrow (G,A,x_{0})$ for $m\geq 2.$ If $n=2$ then $\overline{\pi}_2(G,A)$ is a group but may not abelian.  

We aim to prove that there is an exact sequence for relative digraph homotopy groups that is analogous to the exact sequence for relative homotopy groups of pairs of spaces. To do this, we introduce the relative reduced path-digraph $\overline{P}(G,A,x_0)$.

Recall the notion of subdivision class from~\cite{LWYZ24}. If $G$ is a based digraph, the \emph{minimal path} $f_{min}$
of a standard path $f:J_{m}\rightarrow G$ is given by collapsing vertex $j+1$ to vertex $j$ if $f(j)=f(j+1)$. Define an equivalence relation on all the standard paths of $G$ by saying that
two paths $f$ and $f'$ are \emph{subdivision equivalent} if the minimal path $f_{min}$ of
$f$ is the same as the minimal path $f'_{min}$ of $f'$. The subdivision class of $f$ is denoted
by $\langle f\rangle$.

\begin{defn}
Let $(G,A)$ be a based digraph pair. The \emph{relative reduced path-digraph} $\overline{P}(G,A,x_0)$ is the digraph whose vertex set consists of the subdivision classes $\langle\gamma\rangle$ of digraph maps $\gamma:(J_{m},\partial J_{m},m)\rightarrow (G,A,x_0)$, with arrow set given by the condition that $\langle\gamma\rangle\rightarrow \langle\eta\rangle$ if and only if $\gamma\s_1 \eta\,(rel A)$ or $\eta \s_{-1}\gamma\,(rel A)$. The relative reduced path-digraph is pointed, with basepoint~$\langle e_{x_0}\rangle$ is obtained from the constant digraph map $e_{x_0}:(J_2,\partial J_2,2)\rightarrow x_0.$
\end{defn}

If $A=x_{0}$, so the triple $(G,A,x_{0})$ is the same as the pair $(G,x_{0})$, then the definition of the relative reduced path-digraph reproduces the definition of the reduced path-digraph in~\cite{LWYZ24}. In this case $\overline{P}(G,A,x_{0})$ is more simply denoted as $\overline{P}G$. There is an evaluation map
\[e:\overline{P}G\rightarrow G\]
defined by $e(\langle\lambda\rangle)=\lambda(0)$.

\begin{defn}[\cite{LWYZ24}]Let $X$ be a digraph with basepoint $x_0$ and $G$ be a digraph with basepoint $\ast$. For any based digraph map $f:X\rightarrow G$, the \emph{mapping path digraph} $P_f$ is the pullback of $f$ and $e$:
$$
\xymatrix{
  P_f \ar[d]_{f'} \ar[r]_{q} \pushoutcorner
                & \overline{P}G \ar[d]^{e}  \\
  X  \ar[r]_{f}
               & G.             }
$$
More precisely, $P_f$ is the based digraph with vertex set
$$V(P_f)=\{(x,\langle\lambda\rangle)\in V(X)\times V(\overline{P}G)\mid e(\langle\lambda\rangle)=f(x)\}$$
and arrow set given by the condition that $(x,\langle\lambda\rangle) \rightarrow (x',\langle\lambda'\rangle)$ in $P_f$ if and only if one of two possibilities holds:
\begin{itemize}
    \item $x\rightarrow x'$ and $\langle\lambda\rangle \rightarrow \langle\lambda'\rangle$;
    \item $x= x'$ and $ \langle\lambda\rangle \rightarrow \langle\lambda'\rangle$.
\end{itemize}
The digraph $P_{f}$ is pointed with basepoint $(x_0,\langle e_{\ast}\rangle)$.
\end{defn}

\begin{prop}\label{iso}Let $(G,A)$ be a based
 digraph pair. The relative reduced path-digraph $\overline{P}(G,A,x_0)$ is isomorphic to the mapping path digraph $P_{i}$ of the inclusion digraph map $i:A\rightarrow G.$
\begin{proof}
Observe that if $(a,\langle\lambda\rangle)\in V(P_{i})$ then $e(\langle\lambda\rangle))=i(a)$. By definition of the evaluation map, $e(\langle\lambda\rangle)=\lambda(0)$, and as $i$ is an inclusion, $i(a)=a$. Thus $$V(P_i)=\{(a,\langle\lambda\rangle)\}\in V(A)\times V(\overline{P}G)\mid \lambda(0) = a\}.$$ Next observe that $(a,\langle\lambda\rangle)\in V(P_i)$ means that $\langle\lambda\rangle\in V(\overline{P}G)$, so we have $\lambda:(J_{m},m)\rightarrow (G,x_{0})$. The condition that $\lambda(0)=a$ therefore implies that $\lambda(\partial J_m) \subset A$. Thus $\lambda$ may be regarded as a map $\lambda: (J_{m},\partial J_{m},m) \rightarrow (G,A,x_0)$. This lets us define a map $$ \theta: P_{i}\rightarrow \overline{P}(G,A,x_0), \quad \quad \quad (a,\langle\lambda\rangle)\mapsto \langle\lambda\rangle.$$ Let us check $\theta$ is a digraph map. By definition, any arrow $(a,\langle\lambda\rangle) \rightarrow (a',\langle\lambda'\rangle)$ is in $P_i$ if and only if either
\begin{itemize}
    \item $a\rightarrow a'$ and $\langle\lambda\rangle \rightarrow \langle\lambda'\rangle$;
    \item or $a= a'$ and $ \langle\lambda\rangle \rightarrow \langle\lambda'\rangle$.
\end{itemize}
That is, $(a,\langle\lambda\rangle) \rightarrow (a',\langle\lambda'\rangle)$ in $P_i$ if and only if $ \langle\lambda\rangle \rightarrow \langle\lambda'\rangle$. Hence, $\theta$ is a digraph map. Clearly, there is also a digraph map $$\zeta: \overline{P}(G,A,x_0)  \rightarrow P_{i}, \quad \quad \quad   \langle\lambda\rangle \mapsto  (\lambda(0), \langle\lambda\rangle).$$
For any vertex $(a,\langle\lambda\rangle) \in P_i$, $$\zeta\circ \theta (a,\langle\lambda\rangle) = \zeta(\langle\lambda\rangle) = (\lambda(0), \langle\lambda\rangle) = (a,\langle\lambda\rangle) .$$ Therefore, $\zeta\circ \theta =id_{P_i}$. For any vertex $\langle\lambda\rangle \in \overline{P}(G,A,x_0),$ $$ \theta \circ \zeta ( \langle\lambda\rangle) = \theta(\lambda(0),\langle\lambda\rangle) = \langle\lambda\rangle = id_{\overline{P}(G,A,x_0)} .$$ Hence, $\theta: P_i\rightarrow \overline{P}(G,A,x_0)$ is an isomorphism between digraphs. 
\end{proof}
\end{prop}

Taking inspiration from topological spaces, we next explore the relationship between the absolute homotopy group $\pi_{n-1}(\overline{P}(G,A,x_0))$ and the relative homotopy group $\overline{\pi}_n(G,A)$. This first involves constructing a mapping space for digraph triples $(H,B,D)^{(G,A,C)}$.

\begin{defn}
Let $(G,A,C)$ and $(H,B,D)$ be digraph triples.
The \emph{mapping digraph triple} $(H,B,D)^{(G,A,C)} $ is the digraph whose vertex set is the elements of $Hom(\left(G,A,C\right);\left(H,B,D\right))$ and whose arrows $f\rightarrow g$ satisfy  $f(v)\rightarrow g(v)$ or $f(v)= g(v)$ for any $v\in V(G)$.
\end{defn}

Since $\jt = (J_{m_1},\partial J_{m_1},m_1)\Boxx (J_{m_i},\partial J_{m_i})^{\Boxx (n-1)}$, there is a duality map
\begin{small}
$$
  d_{m_{1}}: Hom(\jt;(G,A,x_0))\rightarrow 
    Hom(\jp;((G,A,x_0)^{(J_{m_1},\partial J_{m_1},m_1)},e_{x_0}^{J_{m_1}}))
$$
\end{small}

\noindent 
defined by $d_{m_1}(\gamma)(a_1,a_2,\dots,a_{n-1})(j)=\gamma(j,a_1,a_2,\dots,a_{n-1})$,
where $e^{J_{m_1}}_{x_0}: J_{m_1}\rightarrow x_0.$ Observe that $d_{m_1}$ is an isomorphism because it has an inverse map 
\begin{small}
$$
  d_{m_{1}}^{-1}:Hom(\jp;((G,A,x_0)^{(J_{m_1},\partial J_{m_1},m_1)},e_{x_0}^{J_{m_1}}))\rightarrow 
   Hom(\jt;(G,A,x_0))
$$
\end{small}
\noindent 
defined by $d_{m_1}^{-1}(\eta)(j,a_1,a_2,\dots,a_{n-1})=\eta(a_1,a_2,\dots,a_{n-1})(j)$. 
There is a projection map $$p_{m_1}:((G,A,x_0)^{(J_{m_1},\partial J_{m_1},0)},e_{x_0}^{J_{m_1}}) \rightarrow (\overline{P}(G,A,x_0),\langle e_{x_0}\rangle),\quad \quad \quad \lambda \mapsto \langle \lambda \rangle.$$
Define  
\begin{small} 
$$
   \phi_{m_{1}}: Hom(\jt;(G,A,x_0))\rightarrow
   Hom(\jp;(\overline{P}(G,A,x_0),\langle e_{x_0}\rangle))
   $$
\end{small} 

\noindent 
by the composite $\phi_{m_1} = p_{m_1}\circ d_{m_1}$. Taking the union of all digraph grid maps of any length gives a map 
\begin{small}
\begin{multline*}
 \phi^n  = \bigcup\limits_{m_1\geq 0}\phi_{m_1}: \bigcup\limits_{m_i\geq 0\atop i\geq 1} Hom(\jt;(G,A,x_0)) \rightarrow \\ \bigcup\limits_{m_i\geq 0 \atop\text{ }i\geq 1} Hom(\jp;(\overline{P}(G,A,x_0),\langle e_{x_0}\rangle)).
\end{multline*}
\end{small}
\begin{lem}\label{surjec}Let $(G,A)$ be a based digraph pair. For $n\geq 1$, the map $\phi^n$ is surjective.
\begin{proof}
Let $\widetilde{f}\in Hom(\jp;$ $(\overline{P}(G,A,x_0),\langle e_{x_0}\rangle))$, so $\widetilde{f}$ is a map
$\widetilde{f}: \jp\rightarrow (\overline{P}(G,A,x_0),\langle e_{x_0}\rangle)$. Let $a\in \jp$. Then $\widetilde{f}(a)$ is a map $\widetilde{f}(a): (J_{m},\partial J_{m},m) \rightarrow (G,A,x_0)$.

If $n=1$ then $(J_{m_i},\partial J_{m_i})^{\Box 0}$ is a discrete vertex and any digraph map $f: (J_m,\partial J_m,m)\rightarrow (G,A,x_0)$ is a vertex in $\overline{P}(G,A,x_0)$, so the lemma holds.
If $n=2$, the lemma holds by adjusting the length of representations in subdivision classes of paths.
%

Suppose that for $n<k$ the lemma holds. Let us check it holds for $n=k$. Given a map  $\widetilde{f}: (J_{m_{i}},\partial J_{m_{i}})^{\Box (k-2)} \Boxx (J_{m_{k-1}},\partial J_{m_{k-1}})\rightarrow (\overline{P}(G,A,x_0),\langle e_{x_0}\rangle).$ Fixing $j$ in $J_{m_{k-1}},$ let $$\widetilde{f}_j:= \widetilde{f}|_{(J_{m_{i}},\partial J_{m_{i}})^{\Box (k-2)}\Box j}:(J_{m_{i}},\partial J_{m_{i}})^{\Box (k-2)}\Boxx j\rightarrow (\overline{P}(G,A,x_0), \langle e_{x_0}\rangle).$$ Then there is a digraph map $$f_j: (J_{m_{k }^j},\partial J_{m_{k}^j},m_{k}^j)\Boxx (J_{m_{i}},\partial J_{m_{i}})^{\Box (k-2)}\Boxx j\rightarrow (G, A,x_0)$$ such that $\phi_{m_{k}^j}(f_j)=\widetilde{f}_j$ since the lemma holds for $n=k-1$. 
Similarly, for the same fixed $j$ in $J_{m_{k-1}}$, let 
$$\widetilde{f}_{j+1}:= \widetilde{f}|_{(J_{m_{i}},\partial J_{m_{i}})^{\Box (k-2)}\Box (j+1)}:(J_{m_{i}},\partial J_{m_{i}})^{\Box (k-2)}\Boxx (j+1)\rightarrow (\overline{P}(G,A,x_0), \langle e_{x_0}\rangle).$$
Then there is a digraph map $$f_{j+1}: (J_{m_{k}^{j+1}},\partial J_{m_{k}^{j+1}}, m_{k}^{j+1})\Boxx (J_{m_{i}},\partial J_{m_{i}})^{\Box (k-2)}\Boxx (j+1)\rightarrow (G,A,x_0)$$ such that $\phi_{m_{k}^{j+1}}(f_{j+1})=\widetilde{f}_{j+1}$.  

Consider $j\rightarrow j+1$,   (the discussion for $j\leftarrow j+1$ is similar). Repeat the procedure over all $a\neq a_0$ as in~\cite[Lemma 4.17]{LWYZ24} to obtain a digraph map
$$f_{j\rightarrow j+1}: (J_{M_{k}^{j+1}},\partial J_{M_{k}^{j+1}}, M_{k}^{j+1})\Boxx (J_{m_{i}},\partial J_{m_{i}})^{\Box (k-2)}\Boxx (j\rightarrow j+1)\rightarrow (G, A,x_0)$$
 such that $\phi_{M_{k}^{j+1}}(f_{j\rightarrow j+1})=\widetilde{f}_{j\rightarrow j+1},$
where $\widetilde{f}_{j\rightarrow j+1}= \widetilde{f}|_{(J_{m_{i}},\partial J_{m_{i}})^{\Box (k-2)}\Box (j\rightarrow j+1)}.$

Iterating over all $j\in J_{m_{k-1}}$, we finally obtain the digraph map $f$ asserted, where $$f:(J_{m_{k}},\partial J_{m_{k}}, m_{k} ) \Boxx (J_{m_{i}},\partial J_{m_{i}})^{\Box (k-1)}\rightarrow (G,A,x_0)$$ is the digraph map such that $\phi_{m_{k}}(f) = \widetilde{f}$. This proves the lemma.
\end{proof}
\end{lem} 

Taking direct limits of Hom sets, the map $\phi^{n}$ induces a map 
\begin{small} 
$$ 
Hom((J^{\Box n},\partial J^{\Box n},\overline{J}^{\Box (n-1)});(G,A,x_0))\rightarrow 
Hom((J,\partial J)^{\Box(n-1)};(\overline{P}(G,A,x_{0}),\langle e_{x_{0}}\rangle)). 
$$
\end{small} 

\noindent 
Taking homotopy classes then induces a map 
$$\Phi_{n}:[(J^{\Box n},\partial J^{\Box n},\overline{J}^{\Box (n-1)}),(G,A,x_0)]\rightarrow [((J,\partial J)^{\Box(n-1)},(\overline{P}(G,A,x_{0}),\langle e_{x_{0}}\rangle))],$$ 
that is, a map 
$$\Phi_n:\overline{\pi}_{n}(G,A)\rightarrow \overline{\pi}_{n-1}(\overline{P}(G,A,x_0)).$$

\begin{prop}\label{dual}
Let $(G,A)$ be a based digraph pair. For $n\geq 2$, the map $$\overline{\pi}_{n}(G,A)\stackrel{\Phi_n}{\longrightarrow} \overline{\pi}_{n-1}(\overline{P}(G,A,x_0)$$ 
is an isomorphism. 
\begin{proof}
We check properties of $\Phi_{n}$ one at a time.
\medskip

\noindent
\textit{Well-defined}.
Let $$f:\jt\rightarrow (G,A,x_0)$$ and $$g:(J_{l_i}^{\Box n},\partial J_{l_i}^{\Box n},\overline{J}_{l_i}^{\Box (n-1)})\rightarrow (G,A,x_0)$$ represent elements $\{f\}, \{g\}\in Hom((J^{\Box n},\partial J^{\Box n},\overline{J}^{\Box (n-1)});(G,A,x_0)).$ Recall that $\{f\}\s_{1} \{g\}$ or $\{f\}\s_{-1} \{g\}$ if and only if for any representatives $f$ and~$g$ of $\{f\}$ and $\{g\}$ respectively we have $f\s_1 g$ or $f\s_{-1}g$.

Suppose that $f\simeq_{F}g$. Then by definition of an $F$-homotopy there is a sequence of maps $\{f_{l}\}_{l=0}^{s}$ such that $f_{l}\s_1 f_{l+1}$ or $f_{l}\s_{-1} f_{l+1}$. By definition of $\phi^n$ and the fact that $\jt = (J_{m_1},\partial J_{m_1},m_1)\Boxx (J_{m_i},\partial J_{m_i})^{\Boxx (n-1)}$, we need to verify that $ \phi_{m_{1}}(f)\simeq_{F}\phi_{l_{1}}(g)$. If so, then $\Phi_{n}(f)\simeq_{F}\Phi_{n}(g).$ It is therefore sufficient to show that if $f\simeq_{1}g$ or $f\simeq_{-1}g$ then $ \phi_{m_1}(f)\simeq_{F}\phi_{l_1}(g)$. Let   $\widetilde{f}=\phi_{m_1}(f)$ and $\widetilde{g}=\phi_{l_1}(g).$

If $f\simeq_{1}g$, then there exist two subdivisions $$\overline{f}:(J_{M_i}^{\Box n},\partial J_{M_i}^{\Box n},\overline{J}_{M_i}^{\Box (n-1)})\rightarrow (G,A,x_0)$$ and $$\overline{g}:(J_{M_i}^{\Box n},\partial J_{M_i}^{\Box n},\overline{J}_{M_i}^{\Box (n-1)})\rightarrow (G,A,x_0)$$ of $f$ and $g$ respectively such that $\overline{f}\dr \overline{g} $. Thus $\phi_{M_{1}}(\overline{f})=\widetilde{\overline{f}}\dr \widetilde{\overline{g}}=\phi_{M_{1}}(\overline{g})$. Now let us consider the relationship between $\widetilde{\overline{g}}$ and $\widetilde{g}$ in the following three cases.

\begin{enumerate}
  \item If $\overline{g}$ is a subdivision of $g$ in the last $n-1$ coordinates, then $\widetilde{\overline{g}}$ is a subdivision of $\widetilde{g}$. Therefore $\widetilde{\overline{g}}\s_1 \widetilde{g}.$
  \item If $\overline{g}$ is a subdivision of $g$ in the first coordinate, then $\widetilde{\overline{g}}_{a}$ is a subdivision of $\widetilde{g}_a$ for any $a\in (J_{M_{i}},\partial J_{M_{i}})^{\Boxx (n-1)}$, so $\langle\widetilde{\overline{g}}_{a}\rangle = \langle\widetilde{g}_a\rangle$ in $\overline{P}(G,A,x_0)$. Hence $\widetilde{\overline{g}} = \widetilde{g}$.
  \item By the subdivision decomposition, if $\overline{g}$ is a subdivision of $g$ not only in the last $n-1$ coordinates but also in first cooordinate, then $\widetilde{\overline{g}}\s_1 \widetilde{g}$ by considering both cases above.
\end{enumerate}

In conclusion, $\widetilde{\overline{g}}\s_1 \widetilde{g}$ or $\widetilde{\overline{g}}= \widetilde{g}$.  Similarly,  $\widetilde{\overline{f}}\s_1\widetilde{f}$ or $\widetilde{\overline{f}}=\widetilde{f}$. Therefore $\widetilde{f}\s_F \widetilde{g}.$ Hence if $f\simeq_{F}g$ then $\widetilde{f}\simeq_{F}\widetilde{g}$. This implies that $\Phi_{n}$ is well-defined.
\medskip

\noindent
\textit{Surjectivity}. 
This follows form the definition of $\Phi_n$ since each $\phi_n$ is surjective by Lemma~\ref{surjec}. 
\medskip

\noindent
\textit{Injectivity}.
We only need to show that $\{f\}\simeq_{F}\{g\}$ if $[\Phi_n(\{f\})]=[\Phi_n(\{g\})]$. Since $\Phi_n(\{f\}) = \{\phi_{m_1}(f)\}$ and $\Phi_n(\{g\}) = \{\phi_{l_1}(g)\}$,
if $\Phi_n(\{f\})\simeq_{F}\Phi_n(\{g\})$, then there is a sequence of digraph maps $\{\widetilde{f}_i\}_{i=0}^{l}$ from $\phi_{m_1}(f)$ to $\phi_{l_1}(g)$ such that $\widetilde{f}_{i}\s_1\widetilde{f}_{i+1}$ or $\widetilde{f}_{i}\s_{-1}\widetilde{f}_{i+1}$ for $0\leq i\leq l-1$.

Suppose that $\widetilde{f}_{0}$ and $\widetilde{f}_{1}$ are maps
$$\widetilde{f}_0:(J_{m_{i}},\partial J_{m_i})^{\Boxx (n-1)}\rightarrow (\overline{P}(G,A,x_0), \langle e_{x_0} \rangle)$$
and $$\widetilde{f}_1:(J_{k_{i}},\partial J_{k_i})^{\Boxx (n-1)}\rightarrow (\overline{P}(G,A,x_0),\langle e_{x_0} \rangle).$$
By Lemma~\ref{surjec} there exist maps $$f_0:(J_{m_{1}},\partial J_{m_{1}}, m_{1}) \Boxx(J_{m_{i}},\partial J_{m_i})^{\Boxx  (n-1)}\rightarrow (G,A,x_0)$$ and
$$f_1:(J_{k_{1}},\partial J_{k_{1}}, k_{1}) \Boxx(J_{k_{i}},\partial J_{k_i})^{\Boxx (n-1)}\rightarrow (G,A,x_0)$$ such that
$\phi_{m_1}(f_0) = \widetilde{f}_0$ and $\phi_{k_{1}}(f_1) = \widetilde{f}_1$. To prove $\Phi_n$ is injective, we only need to prove $f_0\s_1 f_1$ if
$\widetilde{f}_{0}\s_1 \widetilde{f}_1$ and $f_0\s_{-1} f_1$ if $\widetilde{f}_{0}\s_{-1} \widetilde{f}_1$. By definition of a one-step $F$-homotopy, there are subdivisions
$\overline{\widetilde{f}}_0$ of $\widetilde{f}_{0}$ and $\overline{\widetilde{f}}_1$ of
$ \widetilde{f}_1$ by $q_0$ and $q_1$ respectively such that
$\overline{\widetilde{f}}_0 \dr \overline{\widetilde{f}}_1$,
where $$\overline{\widetilde{f}}_0, ~\overline{\widetilde{f}}_1:(J_{M_{i}},\partial J_{M_i})^{\Boxx (n-1)}\rightarrow (\overline{P}(G,A,x_0),\langle e_{x_0}\rangle).$$ 
Thus there exist maps $$\overline{f}_0: (J_{L_{1}},\partial J_{L_{1}},L_{1}) \Boxx (J_{M_{i}},\partial J_{M_i})^{\Boxx (n-1)}\rightarrow (G,A,x_0)$$
 and $$\overline{f}_1: (J_{K_{1}},\partial J_{K_{1}},K_{1}) \Boxx(J_{M_{i}},\partial J_{M_i})^{\Box (n-1)}\rightarrow (G,A,x_0)$$
such that $\phi_{L_1}(\overline{f}_0) = \overline{\widetilde{f}}_0$ and $\phi_{K_{1}}(\overline{f}_1) = \overline{\widetilde{f}}_1$. It is clear that $\overline{f}_0$ and $\overline{f}_1$ are subdivisions of $f_0$ and $f_1$ respectively.

Since $\overline{\widetilde{f}}_0 \dr \overline{\widetilde{f}}_1$, by Lemma \ref{surjec} there is a digraph map $$F: (J_{M_{1}},\partial J_{M_{1}}, M_{1}) \Boxx (J_{M_{i}},\partial J_{M_i})^{\Boxx (n-1)}\Boxx J_1\rightarrow (G,A,x_0) $$ such that $$\phi_{M_{1}}(F)|_{(J_{M_{i}},\partial J_{M_i})^{\Boxx (n-1)}\Boxx 0} = \overline{\widetilde{f}}_0$$ and $$\phi_{M_{1}}(F)|_{(J_{M_{i}},\partial J_{M_i})^{\Boxx (n-1)}\Boxx 1} = \overline{\widetilde{f}}_1.$$
Denoting $$\overline{f}_{0}':=F|_{ (J_{M_{1}},\partial J_{M_{1}},M_{1}) \Box(J_{M_{i}},\partial J_{M_i})^{\Boxx (n-1)}\Boxx 0}$$ and $$\overline{f}_1':=F|_{(J_{M_{1}},\partial J_{M_{1}},M_{1}) \Box (J_{M_{i}},\partial J_{M_i})^{\Boxx (n-1)}\Boxx 1},$$ we obtain $\overline{f}_0'\dr \overline{f}_1'$.

Next, we claim that $\overline{f}_0'$ is a subdivision of $\overline{f}_0$. For any $a\in (J_{M_{i}},\partial J_{M_i})^{\Boxx (n-1)}$ and maps  $$(\overline{f}_{0})_a: (J_{L_{1}},\partial J_{L_{1}},L_{1})\rightarrow(G,A,x_0)\quad\mbox{and}\quad(\overline{f}'_{0})_a: (J_{M_{1}},\partial J_{M_{1}}, M_1)\rightarrow(G,A,x_0) $$ satisfying $\langle(\overline{f}'_{0})_a\rangle= \langle(\overline{f}_{0})_a\rangle$, there is a common subdivision by some shrinking map $q:(J_{M_{1}},\partial J_{M_1},M_1)\rightarrow (J_{L_{1}},\partial J_{L_1},L_1)$ for all $a$. Hence $\overline{f}'_0$ is the subdivision of $\overline{f}_0$ by $q \Boxx id^{\Boxx (n-1)}$. As $\overline{f}_0$ is a subdivision of $f_0$, we have shown that $\overline{f}'_0$ is a subdivision of $f_0$. Similarly, $\overline{f}'_1$ is a subdivision of $f_1$. Thus $f_0\s_1 f_1.$ Similarly, $f_0\s_{-1} f_1$ if $\widetilde{f}_{0}\s_{-1} \widetilde{f}_1$. It follows that if $\widetilde{f} = \phi_{m_1}(f)\simeq_{F}\phi_{l_{m_1}}(g) = \widetilde{g}$, then $f\simeq_{F}g$. Hence $\Phi_n$ is injective.
\medskip

\noindent
\textit{Homomorphism}. It is sufficient to show $[\{\phi_{M_{1}}(f\cdot g)\}]= [\{\phi_{m_1}(f)\}\cdot\{\phi_{l_{1}}(g)\}],$ where $M_{1} = max\{m_{1},~l_{1}\}$.
By the definition of the multiplication,
\begin{small} 
$$
f\cdot g =\widetilde{f}\vee \widetilde{g}: (J_{M_{1}},\partial J_{M_{1}}, M_{1}) \Boxx  (J_{(m_{2}+l_2)},\partial J_{(m_{2}+l_2)})\Boxx (J_{M_i},\partial J_{M_i})^{\Boxx (n-2)}) \rightarrow  
(G,A,x_0), 
$$
\end{small} 

\noindent 
where $\widetilde{f}$ and $\widetilde{g}$ are the respective subdivisions of $f$ and $g$ obtained by extending their domains and $M_i=max\{m_i,~l_i\}$ for $3\leq i\leq n$.
Then
$$  \phi_{M_{1}}(f\cdot g)(i_2,i_3,...,i_{n})= \left\{
                     \begin{array}{ll}
                      \phi_{M_{1}}(\widetilde{f})(i_2,i_3,...,i_n), & \hbox{$i_2\leq m_2$;} \\
                        \phi_{M_{1}}(\widetilde{g})(i_2-m_2,i_3,...,i_n), & \hbox{$i_2> m_2$;}
                     \end{array}
                   \right
.$$
and
$$  (\phi_{m_1}(f)\cdot \phi_{l_{1}}(g))(i_2,i_3,...,i_{n})= \left\{
                     \begin{array}{ll}
                       (\widetilde{\phi_{m_{1}}(f)})(i_2,i_3,...,i_n), & \hbox{$i_2\leq m_2$;} \\
                       (\widetilde{\phi_{l_{1}}(g)})(i_2-m_2,i_3,...,i_n), & \hbox{$i_2> m_2$.}
                     \end{array}
                   \right
.$$
For any $(i_2,i_3,...,i_n)\in (J_{(m_{2}+l_2)},\partial J_{(m_{2}+l_2)})\Boxx (J_{M_i},\partial J_{M_i})^{\Boxx (n-2)}$, we have $$\phi_{M_{1}}(f\cdot g)(i_2,i_3,...,i_{n})=(\phi_{m_1}(f)\cdot \phi_{l_{1}}(g))(i_1,i_2,...,i_{n}).$$ Therefore $\{\phi_{M_{1}}(f\cdot g)\}= \{\phi_{m_1}(f)\}\cdot\{\phi_{l_{1}}(g)\}.$ Thus $\Phi_n(f\cdot g) = \Phi_n(f)\cdot \Phi_n(g),$ implying that $\Phi_{n}$ is a homomorphism.
\medskip

Collectively, this shows that $\Phi_{n}$ is an isomorphism.
\end{proof}
\end{prop}

Now we can prove Theorem~\ref{introexact}, that there is a long exact sequence of relative homotopy groups.
By~\cite[Theorem 5.9]{LWYZ24}, for any based digraph map
$f:X\rightarrow G$
there is a long exact sequence
$$\xymatrix@C=0.5cm{
 \cdots\ar[r]&\overline{\pi}_{n+1}(P_{i})\ar[r]^{f'_{n+1}} &\overline{\pi}_{n+1}(X) \ar[r]^{f_{n+1}}&  \overline{\pi}_{n+1}(G) \ar[r]^{\delta_{n}}&
    \overline{\pi}_{n}(P_{f})\ar[r]^{f'_{n}}&
  \overline{\pi}_n(X) \ar[r]^{f_{n}} & \overline{\pi}_n(G)}$$
of based sets for $n\geq 0$, and if $n\geq 1$ then this is a long exact sequence of groups.

In the case of a digraph pair $(G,A)$, apply this to the digraph inclusion
$i: A\rightarrow G$. By Proposition~\ref{iso}, there is a digraph isomorphism
$$\zeta \colon \overline{P}(G,A,x_{0}) \rightarrow P_{i},\quad\quad \langle \gamma \rangle \mapsto (\gamma(0), \langle \gamma \rangle)$$ and by
Proposition~\ref{dual} there is a group isomorphism
$$\Phi_{n+1}:\overline{\pi}_{n+1}(G,A)\rightarrow \overline{\pi}_{n}(\overline{P}(G,A,x_0)),\quad \quad [f]\mapsto [\widetilde{f}],$$ where $f: (J_{m_1},\partial J_{m_1},m_1)\Box (J_{m_i}^{\Box n}, \partial J_{m_i}^{\Box n}) \rightarrow (G,A,x_0)$ and $$\widetilde{f}: (J_{m_i}, \partial J_{m_i})^{\Box n} \rightarrow (G,A,x_0),\quad\quad a\mapsto f|_{J_{m_1\Box a}} .$$
Let $\partial_{n+1}$ be the composite
\begin{equation} 
\label{partialdef} 
\partial_{n+1}:\overline{\pi}_{n+1}(G,A)\stackrel{\Phi_{n+1}}{\longrightarrow}\overline{\pi}_{n}(\overline{P}(G,A,x_{0}))\stackrel{\zeta_{n}}{\longrightarrow}\overline{\pi}_{n}(P_{i})\stackrel{i'_{n}}{\longrightarrow}\overline{\pi}_{n}(A), 
\end{equation} where $i'_n$ is induced by digraph map $$i': P_i \rightarrow A,\quad\quad (\gamma(0),\langle \gamma\rangle) \mapsto\gamma(0)$$
and let $j_{n+1}$ be the composite
\begin{equation} 
\label{jdef} 
j_{n+1}:\overline{\pi}_{n+1}(G)\stackrel{\delta_{n}}{\longrightarrow}\overline{\pi}_{n}(P_{i})\stackrel{\zeta_{n}^{-1}}{\longrightarrow}\overline{\pi}_{n}(\overline{P}(G,A,x_{0}))\stackrel{\Phi_{n+1}^{-1}}{\longrightarrow}\overline{\pi}_{n+1}(G,A). 
\end{equation} 
Note that both $\partial_{n+1}$  and  $j_{n+1}$ are group homomorphisms if $n\geq 1$ since
each of $\Phi_{n+1}$, $\zeta_{n}$, $\delta_{n+1}$ and $i_{n}'$ is. We obtain the following, which restates Theorem~\ref{introexact}.

\begin{thm}\label{exact}
Let $(G,A)$ be a based digraph pair. Then there is a long exact sequence
$$\xymatrix@C=0.5cm{
 \cdots\ar[r]&\overline{\pi}_{n+2}(G,A)\ar[r]^-{\partial_{n+2}} &\overline{\pi}_{n+1}(A) \ar[r]^{i_{n+1}}&  \overline{\pi}_{n+1}(G) \ar[r]^-{j_{n+1}}&
    \overline{\pi}_{n+1}(G,A)\ar[r]^-{\partial_{n+1}}&
  \overline{\pi}_n(A) \ar[r]^{i_{n}} & \overline{\pi}_n(G)}$$
of based sets for any $n\geq0 $. If $n\geq 1$, it is a long exact sequence of groups.~$\qqed$
\end{thm}


Further, a map of digraph triples induces a commutative diagram of long exact sequences of relative homotopy groups.

\begin{prop}\label{mor}
Any map $f\colon(G,A,x_0)\rightarrow (H,B,y_0)$ of digraph triples induces a commutative diagram of exact sequences
$$\xymatrix@C=0.5cm{
 \cdots\ar[r]&\overline{\pi}_{n+2}(G,A)\ar[r]^-{\partial_{n+2}} \ar[d] &\overline{\pi}_{n+1}(A) \ar[r]^{i_{n+1}}\ar[d] &  \overline{\pi}_{n+1}(G) \ar[r]^-{j_{n+1}}\ar[d] &
    \overline{\pi}_{n+1}(G,A)\ar[r]^-{\partial_{n+1}}\ar[d]&
  \overline{\pi}_n(A) \ar[r]^{i_{n}}\ar[d]  & \overline{\pi}_n(G)\ar[d] \\
   \cdots\ar[r]&\overline{\pi}_{n+2}(H,B)\ar[r]^-{\partial_{n+2}} &\overline{\pi}_{n+1}(B) \ar[r]^{i_{n+1}} &  \overline{\pi}_{n+1}(H) \ar[r]^-{j_{n+1}}&
    \overline{\pi}_{n+1}(H,B)\ar[r]^-{\partial_{n+1}}&
  \overline{\pi}_n(B) \ar[r]^{i_{n}}  & \overline{\pi}_n(H).
  }$$
\begin{proof}
By~\cite[Proposition 5.10]{LWYZ24}, a commutative diagram of based digraphs
$$\xymatrix{
     X\ar[d]^{u}\ar[r]^{i} & G\ar[d]^{f} \\
     Y\ar[r]^{g} & H   }$$
induces a commutative diagram of exact sequences
$$\xymatrix@C=0.5cm{
 \cdots\ar[r]&\overline{\pi}_{n+1}(P_{f})\ar[r]^-{f'_{n+1}} \ar[d] &\overline{\pi}_{n+1}(X) \ar[r]^{f_{n+1}}\ar[d] &  \overline{\pi}_{n+1}(G) \ar[r]^-{\delta_{n}}\ar[d] &
    \overline{\pi}_{n}(P_{f})\ar[r]^-{f'_{n}}\ar[d]&
  \overline{\pi}_n(X) \ar[r]^{f_{n}}\ar[d]  & \overline{\pi}_n(G)\ar[d] \\
   \cdots\ar[r]&\overline{\pi}_{n+1}(P_{f})\ar[r]^-{g'_{n+1}} &\overline{\pi}_{n+1}(Y) \ar[r]^{g_{n+1}} &  \overline{\pi}_{n+1}(H) \ar[r]^-{\delta_{n}}&
    \overline{\pi}_{n}(P_{g})\ar[r]^-{g'_{n}}&
  \overline{\pi}_n(Y) \ar[r]^{g_{n}}  & \overline{\pi}_n(H).
  }$$

In our case, the digraph map $f: (G,A,x_0)\rightarrow (H,B,y_0)$ is equivalent to a commutative diagram of based digraphs
 $$\xymatrix{
  A \ar[d]_{f} \ar[r]^{i}
                & G \ar[d]^{f}  \\
  B  \ar[r]^{j}
                & H.             } $$
Thus there is a commutative diagram of exact sequences
$$\xymatrix@C=0.5cm{
 \cdots\ar[r]&\overline{\pi}_{n+1}(P_{i})\ar[r]^-{i'_{n+1}} \ar[d] &\overline{\pi}_{n+1}(A) \ar[r]^{i_{n+1}}\ar[d] &  \overline{\pi}_{n+1}(G) \ar[r]^-{\delta_{n}}\ar[d] &
    \overline{\pi}_{n}(P_{i})\ar[r]^-{i'_{n}}\ar[d]&
  \overline{\pi}_n(A) \ar[r]^{i_{n}}\ar[d]  & \overline{\pi}_n(G)\ar[d] \\
   \cdots\ar[r]&\overline{\pi}_{n+1}(P_{j})\ar[r]^-{j'_{n+1}} &\overline{\pi}_{n+1}(B) \ar[r]^{j_{n+1}} &  \overline{\pi}_{n+1}(H) \ar[r]^-{\delta_{n}}&
    \overline{\pi}_{n}(P_{j})\ar[r]^-{j'_{n}}&
  \overline{\pi}_n(B) \ar[r]^{j_{n}}  & \overline{\pi}_n(H).
  }$$
The definitions of $\partial_{n+1}$ and $j_{n+1}$ imply that this commutative diagram can be changed into the one in the statement of the proposition provided there is a commutative diagram of digraphs
$$\xymatrix@C=0.5cm{
    \overline{\pi}_{n}(P_{i})\ar[r]^-{\theta_{n}}\ar[d]&
  \overline{\pi}_n(\overline{P}(G,A,x_{0})) \ar[rr]^{\Phi_{n+1}^{-1}}\ar[d]  & & \overline{\pi}_{n+1}(G,A)\ar[d] \\
    \overline{\pi}_{n}(P_{j})\ar[r]^-{\theta_{n}}&
  \overline{\pi}_n(\overline{P}(H,B,y_{0})) \ar[rr]^{\Phi_{n+1}^{-1}}  & & \overline{\pi}_{n+1}(H,B)
  }$$
and a corresponding commutative diagram of digraphs with respect to $\theta^{-1}_{n}$ and $\Phi_{n+1}$. Since $\theta$ is a map of based digraphs the functoriality property in Proposition~\ref{functoriality} implies that the left square commutes, and similarly for $\theta^{-1}_{n}$. By definition, $\Phi_{n+1}$ is an inverse limit induced by maps $\phi_{n}=p\circ d_{n}$. Here, $p$ is a projection of digraphs, so the functoriality property in Proposition~\ref{functoriality} implies that it commutes with maps of digraphs, while $d_{n}$ is the duality isomorphism of Hom sets, whose definition shows that it commutes with maps of digraphs. Thus the right square of the corresponding diagram with respect to $\Phi_{n+1}$ commutes, as does the right square of the pictured diagram with respect to $\Phi^{-1}_{n+1}$. 
  \end{proof}
\end{prop}

\section{Digraph Hurewicz and suspension homomorphisms}
In this section we define a digraph Hurewicz homomorphism and, using the exact sequence of relative digraph homotopy groups in Theorem~\ref{exact}, show that there is a digraph suspension homomorphism. The two homomorphisms are then shown to commute in a precise way. 

\subsection{The digraph Hurewicz homomorphism} We first construct a Hurewicz homomorphism  $$H_n:\bar{\pi}_n(G,A)\rightarrow H_n^c(G,A)$$ from digraph homotopy groups to cubical digraph homology. 

Suppose $(G,A)$ is a digraph pair with base-point $x_0$. For an arbitrary $n$-dimensional cube $I_1^n$, there is an obvious digraph isomorphism  $$\sigma =(\sigma^1,\sigma^2,\cdots,\sigma^n): J_1^n\rightarrow I_1^n; \quad (x_0,x_1,\cdots, x_n)\mapsto  (y_0,y_1,\cdots,y_n),$$ where $\sigma^j :J_1\rightarrow I_1$ is defined by  $$\sigma^j(x_j) = y_j = \left\{ \begin{array}{lll}
 x_j, & \hbox{\text{ if $0\rightarrow 1$ in $j^{th}$-axis of $I_1^n$};} \\
        1- x_j   , & \hbox{$\text{ if $1\rightarrow 0$ in $j^{th}$-axis of $I_1^n$}$.}
 \end{array}\right.$$
Write the inversion number of the $j^{th}$-axis of $\sigma$ by $t(\sigma^j(0),\sigma^j(1))$ and  write the total inversion number of $\sigma$ by $T(\sigma) = \sum\limits_{i=1}^n t(\sigma^j(0),\sigma^j(1))$.

To define the Hurewicz homomorphism on any element $f:(J_{m_i}^{\Box n},\partial J_{m_i}^{\Box n},\bar{J}_{m_i}^{\Box (n-1)}) $ $\rightarrow (G,A,x_0)$ in $\bar{\pi}_n(G,A)$, we decompose $f$ into an $n$-dimensional singular cubical chain $\sum\limits_s a_sf_s$, where each $f_s: J_1^{\Box n}\rightarrow G$. Restrict to the subdigraph $I_1^{\Box n}$ of $J_{m_i}^{\Box n}$ that originates from vertex $(i_1,i_2,\cdots, i_n)$ in $J_{m_i}^{\Box n}$, and denote this subdigraph by $I_{i_1,i_2,\cdots,i_n}^{\Box n}$.  There is a digraph isomorphism $$\sigma_{i_1,i_2,\cdots,i_n}: J_1^{\Box n} \rightarrow I_{i_1,i_2,\cdots,i_n}^{\Box n} $$ from which we obtain a singular $n$-dimensional cubical chain $$f_{i_1,i_2,\cdots,i_n}: J_1^{\Box n} \xrightarrow{\sigma_{i_1,i_2,\cdots,i_n}}I_{i_1,i_2,\cdots,i_n}^{\Box n}{i}\stackrel{i}{\longrightarrow} J_{m_i}^{\Box n}\stackrel{f}{\longrightarrow}  G.$$
Let 
 $$h_n: Hom((J,\partial J,\overline{J})^{\Box n},(G,A,x_0))\rightarrow C_n^c(G)/C_n^c(A)$$ be a homomorphism defined by   
 $$\{f\}\mapsto [\sum \limits_{0\leq i_k\leq m_k-1\atop
1\leq k\leq n}(-1)^{T(\sigma_{i_1,i_2,\cdots,i_n})}f_{i_1,i_2,\cdots, i_n}]$$ where on the right we have chosen a representative of the class $\{f\}$. We will show $h_n$ is well-defined in Lemma \ref{subind}, that is, the image of $h_n$ is independent of choice of representative $f$ in $\{f\}.$ 
Granting this, $h_n$ then induces a Hurewicz map  $$H_n: \bar{\pi}_n(G,A)\rightarrow H_n^c(G,A);\quad \quad [\{f\}]\mapsto [h_n(\{f\})]$$ between digraph homotopy groups and cubical digraph homology groups for $n\geq1$. .
\begin{lem}\label{subind}
The homomorphism $h_n$ is well-defined for $n\geq 1.$
\begin{proof}
To prove $h_n$ is well-defined, we need to check it is independent of subdivision. That is, we need to check that for any shrinking map  $h:J_{M_i}^n\rightarrow J_{m_i}^n$ and digraph map $f:(J_{m_i}^{\Box n},\partial J_{m_i}^{\Box n},\bar{J}_{m_i}^{\Box (n-1)})\rightarrow (G,A,x_0)$ we have $$[\sum \limits_{0\leq i_k\leq m_k-1\atop
1\leq k\leq n}(-1)^{T(\sigma_{i_1,i_2,\cdots,i_n})}f_{i_1,i_2,\cdots, i_n}] = [\sum \limits_{0\leq i_k\leq M_k-1\atop
1\leq k\leq n}(-1)^{T(\sigma_{i_1,i_2,\cdots,i_n})}(h\circ f)_{i_1,i_2,\cdots, i_n}].$$

Since any $n$-dimensional shrinking map $h$ can be decomposed as an $n$-fold box product of 1-dimensional shrinking maps~\cite{LWYZ24}, i.e. $h = h_1\Boxx h_2\Boxx \cdots\Boxx h_n$ where each $h_i$ is the composition of one-step subdivisions, it suffices to prove $$[\sum \limits_{0\leq i_k\leq m_k-1\atop
1\leq k\leq n}(-1)^{T(\sigma_{i_1,i_2,\cdots,i_n})}f_{i_1,i_2,\cdots, i_n}] = [\sum \limits_{\substack{0\leq i_1\leq m_1\\0\leq i_k\leq m_k-1\\
2\leq k\leq n}}(-1)^{T(\sigma_{i_1,i_2,\cdots,i_n})}(\bar{h}\circ f)_{i_1,i_2,\cdots, i_n}]$$ where $$\bar{h} = h_1\Boxx id \Boxx \cdots\Boxx id: J_{m_1+1}\Boxx J_{m_2}\Boxx\cdots \Boxx J_{m_n}\rightarrow J_{m_i}^{\Box n}.$$

Let $\bar{f}=f\circ \bar{h}$. Suppose that $h_1:J_{m_1+1}\rightarrow J_{m_1}$ is the one-step subdivision of $id: J_{m_1}\rightarrow J_{m_1}$ at vertex $s_0$, i.e. $h_1(s_0)=h_1(s_0+1) $. Observe that: 
\begin{align*} 
& \sum \limits_{0\leq i_1\leq m_1\atop 0\leq i_k\leq m_k-1;
2\leq k\leq n}(-1)^{T(\sigma_{i_1,i_2,\cdots,i_n})}\bar{f}_{i_1,i_2,\cdots, i_n} \\
&= \sum \limits_{0\leq i_k\leq m_k-1\atop
2\leq k\leq n}(-1)^{T(\sigma_{i_2,\cdots,i_n})}\left(\sum\limits_{0\leq i_1\leq m_1}(-1)^{t(\sigma^1_{i_1,i_2,\cdots,i_n}(0),\sigma^1_{i_1,i_2,\cdots,i_n}(1))}\bar{f}_{i_1,i_2,\cdots, i_n} \right) \\
&= \sum \limits_{0\leq i_k\leq m_k-1\atop
2\leq k\leq n}(-1)^{T(\sigma_{i_2,\cdots,i_n})}\bigg(\sum\limits_{0\leq i_1\leq s_0-1}(-1)^{t(\sigma^1_{i_1,i_2,\cdots,i_n}(0),\sigma^1_{i_1,i_2,\cdots,i_n}(1))}\bar{f}_{i_1,i_2,\cdots, i_n}\\
& \hspace{10mm}+ (-1)^{t(\sigma^1_{s_0,i_2,\cdots,i_n}(0),\sigma^1_{s_0,i_2,\cdots,i_n}(1))}\bar{f}_{s_0,i_2,\cdots, i_n} \\ 
& \hspace{10mm}+ \sum\limits_{s_0+1\leq i_1\leq m_1}(-1)^{t(\sigma^1_{i_1,i_2,\cdots,i_n}(0),\sigma^1_{i_1,i_2,\cdots,i_n}(1))}\bar{f}_{i_1,i_2,\cdots, i_n} \bigg)
\end{align*}
\begin{align*}
&= \sum \limits_{0\leq i_k\leq m_k-1\atop
2\leq k\leq n}(-1)^{T(\sigma_{i_2,\cdots,i_n})}(\sum\limits_{0\leq i_1\leq s_0-1}(-1)^{t(\sigma^1_{i_1,i_2,\cdots,i_n}(0),\sigma^1_{i_1,i_2,\cdots,i_n}(1))}\bar{f}_{i_1,i_2,\cdots, i_n})\\
& \hspace{10mm} +\sum \limits_{0\leq i_k\leq m_k-1\atop
2\leq k\leq n}(-1)^{T(\sigma_{i_2,\cdots,i_n})}(\sum\limits_{s_0+1\leq i_1\leq m_1}(-1)^{t(\sigma^1_{i_1,i_2,\cdots,i_n}(0),\sigma^1_{i_1,i_2,\cdots,i_n}(1))}\bar{f}_{i_1,i_2,\cdots, i_n} )\\
& \hspace{10mm}+ \sum \limits_{0\leq i_k\leq m_k-1\atop
2\leq k\leq n}(-1)^{T(\sigma_{i_2,\cdots,i_n})}((-1)^{t(\sigma^1_{i_1,i_2,\cdots,i_n}(0),\sigma^1_{i_1,i_2,\cdots,i_n}(1))}\bar{f}_{s_0,i_2,\cdots, i_n})\\
&= \sum \limits_{0\leq i_k\leq m_k-1\atop
1\leq k\leq n}(-1)^{T(\sigma_{s_1,i_2,\cdots,i_n})}f_{s_0,i_2,\cdots, i_n}\\
& \hspace{10mm}+ \sum \limits_{0\leq i_k\leq m_k-1\atop
2\leq k\leq n}(-1)^{T(\sigma_{i_2,\cdots,i_n})}((-1)^{t(\sigma^1_{s_0,i_2,\cdots,i_n}(0),\sigma^1_{i_1,i_2,\cdots,i_n}(1))}\bar{f}_{s_0,i_2,\cdots, i_n}).
\end{align*}
 Since $h_1$ is the one-step subdivision of $id: J_{m_1}\rightarrow J_{m_1}$ at vertex $s_0$, $\bar{f}_{s_0,i_2,\cdots, i_n}$ is a degeneracy. Thus $\bar{f}_{s_0,i_2,\cdots, i_n}$ lies in $[0]\in C_n^c(G)/C_n^c(A)$. Therefore $$[\sum \limits_{0\leq i_k\leq m_k-1\atop
1\leq k\leq n}(-1)^{T(\sigma_{i_1,i_2,\cdots,i_n})}f_{i_1,i_2,\cdots, i_n} ]= [\sum \limits_{\substack{0\leq i_k\leq M_k-1\\
1\leq k\leq n}}(-1)^{T(\sigma_{i_1,i_2,\cdots,i_n})}(\bar{h}\circ f)_{i_1,i_2,\cdots, i_n}].$$ Inductively, we have for any shrinking map $h$, $$[\sum \limits_{0\leq i_k\leq m_k-1\atop
1\leq k\leq n}(-1)^{T(\sigma_{i_1,i_2,\cdots,i_n})}f_{i_1,i_2,\cdots, i_n} ]= [\sum \limits_{\substack{0\leq i_k\leq M_k-1\\
1\leq k\leq n}}(-1)^{T(\sigma_{i_1,i_2,\cdots,i_n})}(h\circ f)_{i_1,i_2,\cdots, i_n}].$$ Hence, $h_n$ is well-defined.
\end{proof}
\end{lem}
\begin{thm}\label{Hure}
Let $G$ be a digraph with basepoint $x_0.$ Then $$H_n: \bar{\pi}_n(G,A)\rightarrow H_n^c(G,A)$$ is a map of sets for $n\geq 1$ and a group homomorphism for $n\geq 2$.
\begin{proof}
Let $f,\text{ }g:(J_{m_i}^{\Box n},\partial J_{m_i}^{\Box n},\bar{J}_{m_i}^{\Box (n-1)})\rightarrow (G,A,x_0)$ be digraph maps. To show that $H_n$ is well-defined  we need to check
 $$ \partial^c_{n}\circ h_{n}(\{f\})=\partial^c_{n}([\sum \limits_{0\leq i_k\leq m_k-1\atop
1\leq k\leq n}(-1)^{T(\sigma_{i_1,1_2,\cdots,i_n})}f_{i_1,i_2,\cdots, i_n}]) \in C_{n-1}^c(A)$$
and if $\{f\}\simeq_F \{g\}$ then 
$$[\sum \limits_{0\leq i_k\leq m_k-1\atop
1\leq k\leq n}(-1)^{T(\sigma_{i_1,1_2,\cdots,i_n})}f_{i_1,i_2,\cdots, i_n}] = [\sum \limits_{0\leq j_k\leq l_k-1\atop
1\leq k\leq n}(-1)^{T(\sigma_{j_1,j_2,\cdots,j_n})}g_{j_1,j_2,\cdots, j_n}].$$ 
\medskip 

\noindent 
\textit{Well-definedness under the boundary map}. Using the definitions of $h_n(\{f\})$, $\partial^{c}_{n}$, $d^{c}_{s}$, the face maps $F_{s,0}$ and $F_{s,1}$, the fact $\partial^{c}_{n}$ is a homomorphism and properties of the inversion number, there is a sequence of equalities:

\begin{align*}
& \partial^{c}_{n}(h_n(\{f\})) \\ 
& = \partial^{c}_{n}([\sum \limits_{0\leq i_k\leq m_k-1\atop
1\leq k\leq n}(-1)^{T(\sigma_{i_1,i_2,\cdots,i_n})}f_{i_1,i_2,\cdots, i_n}]) \\
& = [ \sum \limits_{0\leq i_k \leq m_k-1\atop 1\leq k\leq n}(-1)^{T(\sigma_{i_1,i_2,\cdots,i_n})}\partial^{c}_{n}(f_{i_1,i_2,\cdots, i_n})] \\
&= [\sum \limits_{0\leq i_k\leq m_k-1\atop
1\leq k\leq n}(-1)^{T(\sigma_{i_1,i_2,\cdots,i_n})}(\sum \limits_{s=1}^n (-1)^s d_s^cf_{i_1,i_2,\cdots, i_n})] \\
&= [\sum \limits_{0\leq i_k\leq m_k-1\atop
1\leq k\leq n}(-1)^{T(\sigma_{i_1,i_2,\cdots,i_n})}(\sum \limits_{s=1}^n (-1)^s (f_{i_1,i_2,\cdots, i_n}\circ F_{s,0}-f_{i_1,i_2,\cdots, i_n}\circ F_{s,1}))] \\
&= [\sum \limits_{0\leq i_k\leq m_k-1\atop
1\leq k\leq n} \sum \limits_{s=1}^n
(-1)^{T(\sigma_{i_1,i_2,\cdots,i_n})+s}(f_{i_1,i_2,\cdots,  i_n}|_{J_1\Box \cdots \Box\{0\}\Box\cdots \Box J_1}
 -f_{i_1,i_2,\cdots,  i_n}|_{J_1\Box \cdots \Box\{1\}\Box\cdots \Box J_1})]\\
&= [\sum\limits_{s=1}^{n} \sum\limits_{\substack{0\leq i_k\leq m_k-1 \\ 1\leq k\leq n \\ k \neq s}}
 (-1)^{T(\sigma_{i_1,i_2,\cdots,i_{s-1},i_{s+1},\cdots,i_n})+s}
(\sum\limits_{i_s =0}^{m_s -1} (-1)^{2t(\sigma^{s}_{i_1,i_2,\cdots,i_n}(0),\sigma^{s}_{i_1,i_2,\cdots,i_n}(1))} \\
& \hspace{10mm} (f_{i_1,i_2,\cdots,  i_n}|_{J_1\Box \cdots \Box\{i_s\}\Box\cdots \Box J_1}-f_{i_1,i_2,\cdots,  i_n}|_{J_1\Box \cdots \Box\{i_s+1\}\Box\cdots \Box J_1}))]
\end{align*}
\begin{align*}
 &=  [\sum \limits_{s=1}^{n}  \sum\limits_{\substack{0\leq i_k\leq m_k-1 \\
1\leq k\leq n \\ k \neq s}}
(-1)^{T(\sigma_{i_1,i_2,\cdots,i_{s-1},i_{s+1},\cdots,i_n})+s} \\
& \hspace{10mm} (\sum\limits_{i_s =0}^{m_s -1}
(f_{i_1,i_2,\cdots,  i_n}|_{J_1\Box \cdots \Box\{i_s\}\Box\cdots \Box J_1}-f_{i_1,i_2,\cdots,  i_n}|_{J_1\Box \cdots \Box\{i_s+1\}\Box\cdots \Box J_1}))]\\
 &= [\sum\limits_{s=1}^n  \sum\limits_{\substack{0\leq i_k\leq m_k-1\\
1\leq k\leq n \\  k \neq s}}
(-1)^{T(\sigma_{i_1,i_2,\cdots,i_{s-1},i_{s+1},\cdots,i_n})+s} \\
& \hspace{10mm}(f_{i_1,i_2,\cdots,  i_n}|_{J_1\Box \cdots \Box\{0\}\Box\cdots \Box J_1}-f_{i_1,i_2,\cdots,  i_n}|_{J_1\Box \cdots \Box\{m_s\}\Box\cdots \Box J_1})].
 \end{align*}  
Since $f_{i_1,i_2,\cdots,  i_n}|_{J_1\Box \cdots \Box\{0\}\Box\cdots \Box J_1}$ and $f_{i_1,i_2,\cdots,  i_n}|_{J_1\Box \cdots \Box\{m_s\}\Box\cdots \Box J_1}$ lie in the boundary of $J_{m_i}^{\Box n}$, we have
 $$f_{i_1,i_2,\cdots,  i_n}|_{J_1\Box \cdots \Box\{0\}\Box\cdots \Box J_1}, \quad f_{i_1,i_2,\cdots,  i_n}|_{J_1\Box \cdots \Box\{m_s\}\Box\cdots \Box J_1}\in C_{n-1 }^c(A).$$ Hence,
   $\partial^c_n([\sum \limits_{0\leq i_k\leq m_k-1\atop
1\leq k\leq n}(-1)^{T(\sigma_{i_1,i_2,\cdots,i_n})}f_{i_1,i_2,\cdots, i_n}])=[0].$ 
\medskip 

\noindent 
\textit{Well-definedness under homotopy}. Since $\{f\}\s_{1}\{g\}$ if and only if
$f\s_1 g$. Thus we will ignore the subdivision class $\{f\}$ and directly consider it as $f$ later.
It suffices to check a $1$-step $F$ homotopy: if $f\simeq_1 g$ then $H_n(f) =[h_n(f)] = [h_n(g)] = H_n(g) $. Suppose there are two shrinking maps $h:J_{M_i}^{\Box n}\rightarrow J_{m_i}^{\Box n}$ and $h^{'}:J_{M_i}^{\Box n}\rightarrow J_{l_i}^{\Box n}$ such that $f\circ h \dr g\circ h^{'}$. By the definition of $1$-step $F$-homotopy, there is a digraph map $F: J_{M_i}^{\Box n} \Box J_1\rightarrow G$ such that $F|_{J_{M_i}^{\Box n} \Box \{0\}} = f\circ h$ and $F|_{J_{M_i}^{\Box n} \Box \{1\}} = g\circ h'.$ 

Using definitions and properties as in the case of well-definedness under the boundary map, there is a sequence of equalities:
 \begin{align*} & \partial^{c}_{n+1}(h_{n+1}(F))\\
 &= \partial^c_{n+1}([\sum \limits_{0\leq i_k\leq M_k-1 \atop
1\leq k\leq n+1}(-1)^{T(\sigma_{i_1,i_2,\cdots,i_{n},0})}F_{i_1,i_2,\cdots, i_{n},0}]) \\
& =  [\sum \limits_{0\leq i_k\leq M_k-1 \atop
1\leq k\leq n+1}(-1)^{T(\sigma_{i_1,i_2,\cdots,i_{n},0})}(\sum \limits_{s=1}^{n+1} (-1)^s (F_{i_1,i_2,\cdots, i_{n},0}\circ F_{s,0}-F_{i_1,i_2,\cdots,i_{n},0} \circ F_{s,1}))] \\
& = [\sum \limits_{0\leq i_k\leq M_k-1 \atop
1\leq k\leq n+1}(-1)^{T(\sigma_{i_1,i_2,\cdots,i_{n},0})}(\sum \limits_{s=1}^{n} (-1)^s (F_{i_1,i_2,\cdots,i_{n},0}\circ F_{s,0}-F_{i_1,i_2,\cdots, i_{n},0} \circ F_{s,1})] \\
& \hspace{10mm} + [(-1)^{n+1} (F_{i_1,i_2,\cdots,i_{n},0}\circ F_{n+1,0}-F_{i_1,i_2,\cdots, i_{n},0} \circ F_{n+1,1}))]\\
& = [\sum \limits_{0\leq i_k\leq M_k-1 \atop
1\leq k\leq n}(-1)^{T(\sigma_{i_1,i_2,\cdots,i_{n},0})}(\sum \limits_{s=1}^{n} (-1)^s (F_{i_1,i_2,\cdots,i_{n},0}\circ F_{s,0}-F_{i_1,i_2,\cdots, i_{n},0} \circ F_{s,1}) ]\\
& \hspace{10mm} + [(-1)^{n+1} (F_{i_1,i_2,\cdots,i_{n},0}\circ F_{n+1,0}-F_{i_1,i_2,\cdots, i_{n},0} \circ F_{n+1,1})) ]\\
& =  [\sum \limits_{0\leq i_k\leq M_k-1 \atop
1\leq k\leq n} \sum \limits_{s=1}^{n}(-1)^{T(\sigma_{i_1,i_2,\cdots,i_n,0})+s} (F_{i_1,i_2,\cdots, i_n,0}\circ F_{s,0}-F_{i_1,i_2,\cdots, i_n,0} \circ F_{s,1})] 
\end{align*}
\begin{align*}
& \hspace{10mm} +  [\sum\limits_{0\leq i_k\leq M_k-1\atop
1\leq k\leq n} (-1)^{T(\sigma_{i_1,i_2,\cdots,i_n,0})+n+1}(F_{i_1,i_2,\cdots, i_n,0}\circ F_{n+1,0}-F_{i_1,i_2,\cdots, i_n,0} \circ F_{n+1,1})]\\
& = [\sum \limits_{0\leq i_k\leq M_k-1 \atop
1\leq k\leq n} \sum \limits_{s=1}^{n}(-1)^{T(\sigma_{i_1,i_2,\cdots,i_n,0})+s} (F_{i_1,i_2,\cdots,  i_n,0}|_{J_1\Box \cdots \Box\{0\}\Box\cdots \Box J_1}-F_{i_1,i_2,\cdots,  i_n,0}|_{J_1\Box \cdots \Box\{1\}\Box\cdots \Box J_1})]\\
&\hspace{10mm}   +  [\sum \limits_{0\leq i_k\leq M_k-1 \atop
1\leq k\leq n} (-1)^{T(\sigma_{i_1,i_2,\cdots,i_n,0})+n+1}(F_{i_1,i_2,\cdots, i_n,0}-F_{i_1,i_2,\cdots, i_n,1})]\\
& =  [\sum \limits_{s=1}^{n}  \sum\limits_{\substack{0\leq i_k\leq M_k-1 \\ 
1\leq k\leq n \\ k\neq s}} (-1)^{T(\sigma_{i_1,i_2,\cdots,i_{s-1},i_{s+1},\cdots,i_n,0})+s} \bigg(\sum\limits_{0\leq i_s\leq m_s-1} (-1)^{2t(\sigma^{s}_{i_1,i_2,\cdots,i_n,0}(0),\sigma^{s}_{i_1,i_2,\cdots,i_n,0}(1))}\\
&\hspace{10mm} (F_{i_1,i_2,\cdots,  i_n}|_{J_1\Box \cdots \Box\{i_s\}\Box\cdots \Box J_1}-F_{i_1,i_2,\cdots,  i_n}|_{J_1\Box \cdots \Box\{i_s+1\}\Box\cdots \Box J_1})\bigg) ]\\ 
& \hspace{10mm} +  [\sum \limits_{0\leq i_k\leq M_k-1 \atop
1\leq k\leq n} (-1)^{T(\sigma_{i_1,i_2,\cdots,i_n,0})+n+1} 
(F_{i_1,i_2,\cdots, i_n,0}\circ F_{n+1,0}-F_{i_1,i_2,\cdots, i_n,0} \circ F_{n+1,1})] \\
& =  [\sum \limits_{s=1}^{n}  \sum \limits_{\substack{0\leq i_k\leq M_k-1\\
1\leq k\leq n \\ k\neq s}} (-1)^{T(\sigma_{i_1,i_2,\cdots,i_{s-1},i_{s+1},\cdots,i_n,0})+s} \\
& \hspace{10mm} (F_{i_1,i_2,\cdots,0,\cdots,  i_n}|_{J_1\Box \cdots \Box\{0\}\Box\cdots \Box J_1}-F_{i_1,i_2,\cdots, \{m_s-1\},\cdots,i_n}|_{J_1\Box \cdots \Box\{m_s\}\Box\cdots \Box J_1} ) ]\\
& \hspace{10mm} +  [\sum \limits_{0\leq i_k\leq M_k-1\atop
1\leq k\leq n} (-1)^{T(\sigma_{i_1,i_2,\cdots,i_n,0})+n+1}(F_{i_1,i_2,\cdots, i_n,0}\circ F_{n+1,0}-F_{i_1,i_2,\cdots, i_n,0} \circ F_{n+1,1}) ] \\
& =[\sum \limits_{s=1}^{n}  \sum\limits_{\substack{0\leq i_k\leq M_k-1\\
1\leq k\leq n \\ k\neq s}} (-1)^{T(\sigma_{i_1,i_2,\cdots,i_{s-1},i_{s+1},\cdots,i_n,0})+s} \\
& \hspace{10mm} (F_{i_1,i_2,\cdots,0,\cdots,  i_n}|_{J_1\Box \cdots \Box\{0\}\Box\cdots \Box J_1}-F_{i_1,i_2,\cdots, \{m_s-1\},\cdots,i_n}|_{J_1\Box \cdots \Box\{m_s\}\Box\cdots \Box J_1} ) ] \\
& \hspace{10mm}+ [(-1)^{n+1}(h_n(f\circ h)-h_n(g\circ h'))].
\end{align*}
Hence 
\begin{align*} h_n(f\circ h)-h_n(g\circ h') & = [(-1)^{n+1}(\partial^{c}_{n+1} (h_{n+1}(F))]\\ 
& \hspace{-20mm}-[\sum \limits_{s=1}^{n}  \sum \limits_{\substack{ 0\leq i_k\leq M_k-1 \\
1\leq k\leq n \\ k\neq s}} (-1)^{T(\sigma_{i_1,i_2,\cdots,i_{s-1},i_{s+1},\cdots,i_n,0})} (F_{i_1,i_2,\cdots,0,\cdots, i_n,0}-F_{i_1,i_2,\cdots,m_s,\cdots, i_n,0})]. 
\end{align*} 
Since $$\sum \limits_{s=1}^{n}  \sum \limits_{\substack{0\leq i_k\leq M_k-1 \\
    1\leq k\leq n\\ k\neq s}} (-1)^{T(\sigma_{i_1,i_2,\cdots,i_{s-1},i_{s+1},\cdots,i_n,0})} (F_{i_1,i_2,\cdots,0,\cdots, i_n,0}-F_{i_1,i_2,\cdots,m_s,\cdots, i_n,0})\in C_{n+1}^c(A)$$  
it follows that $$h_n(f\circ h)-h_n(g\circ h') \in \partial_{n+1}^c(C^c_{n+1}(G)/C^c_{n+1}(A)).$$ Thus $H_n([f\circ h]) = [h_n(f\circ h)]=[h_n(g\circ h')] = H_n[g\circ h']$. Then by Lemma \ref{subind}, we have $H_n([\{f\}])= H_n([\{f\circ h\}])  = H_n[\{g\circ h'\}] = H_n([\{g\}])$.
\vspace{5mm} 
\medskip

\noindent
\textit{$H_{n}$ is a homomorphism.}  Now suppose that there are digraph maps 
$$f:(J_{m_{i}}^{\Box n},\partial J_{m_{i}}^{\Box n}, \overline{J}_{m_i}^{\Box (n-1)})\rightarrow (G,A,x_0)\quad \text{and} g: (J_{l_{i}}^{\Box n},\partial J_{l_{i}}^{\Box n}, \overline{J}_{l_i}^{\Box (n-1)})\rightarrow (G,A,x_0).$$ 
We prove that $H_n([f\cdot g]) = H_n([f])\cdot H_n([g])$ for $n\geq 2.$ By Remark~\ref{mujindependence}, $f\cdot g$ is obtained by first adjusting the lengths in each coordinate except the second so they match, and then multiplying in any box coordinate but the first coordinate. With this in mind, let $M_j = max\{m_j,l_j\}$ for $j\neq 2.$ Define 
$$\widetilde{f}: (J_{M_1},\partial J_{M_1}, M_1)\Box (J_{m_2},\partial J_{m_2})\Box (J_{M_i},\partial J_{M_i})^{\Box (n-2)} \rightarrow (G,A,x_0)$$
by $$\widetilde{f}(i_1,i_2,...,i_n) = \left\{
                     \begin{array}{ll}
                       f(i_1,i_2,...,i_n) , & \hbox{$i_j \leq  m_j \text{, }j\geq 1$;} \\
                       x_0, & \hbox{$\text{ else}$,}
                     \end{array}
                   \right
.$$ and define 
$$\widetilde{g}:  (J_{M_1},\partial J_{M_1}, M_1)\Box (J_{l_2},\partial J_{l_2})\Box (J_{M_i},\partial J_{M_i})^{\Box (n-2)} \rightarrow (G,A,x_0)$$ 
by $$\widetilde{g}(i_1,i_2,...,i_n) = \left\{
                     \begin{array}{ll}
                     g(i_1,i_2,...,i_n) , & \hbox{$i_j \leq  l_j, \text{, }j\geq 1 $;} \\
                       x_0 , & \hbox{$\text{ else}$.}
                     \end{array}
                   \right.$$
                   Clearly, $\widetilde{f}$ and $\widetilde{g}$ are the subdivisions of $f$ and $g$ respectively. 
Then 
$$f\cdot g :  (J_{M_1},\partial J_{M_1}, M_1)\Box (J_{m_2+l_2+1},\partial J_{m_2+l_2+1})\Box (J_{M_i},\partial J_{M_i})^{\Box (n-2)} \rightarrow (G,A,x_0)$$ 
is defined by 
$$f\cdot g = \widetilde{f}\vee \widetilde{g} = \left\{
                     \begin{array}{ll}
                      \widetilde{f}(i_1,i_2,...,i_n), & \hbox{$0\leq i_2\leq m_2$;} \\
                        \widetilde{g}(i_1,i_2-m_2,...,i_n), & \hbox{$m_2\leq i_2\leq m_2+l_2+1$.}
                     \end{array}
                   \right
.$$
Observe that: 
\begin{align*} h_{n}(f\cdot g)&= [\sum \limits_{\substack{0\leq i_2\leq m_2+l_2+1\\ 0\leq i_k\leq M_k-1\\
k\neq 2}}(-1)^{T(\sigma_{i_1,i_2,\cdots,i_{n}})}(\widetilde{f}\vee \widetilde{g})_{i_1,i_2,\cdots, i_{n}}] \\
& =[\sum \limits_{\substack{0\leq i_2\leq m_2-1\\ 0\leq i_k\leq M_k-1\\
k\neq 2}}(-1)^{T(\sigma_{i_1,i_2,\cdots,i_n})}
(\widetilde{f}\vee \widetilde{g})_{i_1,i_2,\cdots, i_{n}}] \\
&\hspace{10mm} + [\sum \limits_{\substack{m_2\leq i_2\leq m_2+l_2\\ 0\leq i_k\leq M_k-1\\
k\neq 2}}(-1)^{T(\sigma_{i_1,i_2,\cdots,i_{n}})}(\widetilde{f}\vee \widetilde{g})_{i_1,i_2,\cdots, i_{n}}] \\
&=  [\sum \limits_{\substack{0\leq i_2\leq m_2-1\\ 0\leq i_k\leq M_k-1\\
k\neq 2}}(-1)^{T(\sigma_{i_1,i_2,\cdots,i_n})}
\widetilde{f}_{i_1,i_2,\cdots, i_{n}} + \hspace{-5mm}\sum \limits_{\substack{m_1\leq i_1\leq m_1+l_1\\ 0\leq i_k\leq M_k-1\\
2\leq k\leq n}}(-1)^{T(\sigma_{i_1,i_2,\cdots,i_{n}})}\widetilde{g}_{i_1,i_2,\cdots, i_{n}} ]\\
& = h_n(\widetilde{f}) +  h_n(\widetilde{g})
\end{align*} 
Since $f$ is obtained from $\widetilde{f}$ by composing with a shrinking map, $H_n([\widetilde{f}]) = H_n([f]) $, and similarly $H_n([\widetilde{g}]) = H_n([g]) $. Thus $H_n$ is a group homomorphism for $n\geq 2$.
\end{proof}
\end{thm}

Further, the Hurewicz homomorphism is compatible with digraph maps between digraph pairs, as described in the following proposition.
\begin{prop}\label{cubnat} Any digraph map $f: (G,A)\rightarrow (H,B)$ of digraph pairs induces a commutative diagram 
 $$\xymatrix{
    \overline{\pi}_{n}(G,A) \ar[d]_{H_n} \ar[r]^{f_n} & \overline{\pi}_{n}(H,B) \ar[d]^{H_n}  \\
    H^c_{n}(G,A)  \ar[r]^{f^c_n} & H^c_{n}(H,B),             } $$ 
where $H_n$ is the Hurewicz homomorphism.
\begin{proof}
    By Proposition~\ref{inducedhom}, $f$ induces a map $f_n:  \overline{\pi}_{n}(G,A) \rightarrow  \overline{\pi}_{n}(H,B)$ for $n\geq 1$ and a homomorphism for $n\geq 2$. By \cite[Proposition 10]{GMJ21}, $f$ also induces a homomorphism $f^c_n:  H^c_{n}(G,A) \rightarrow   H^c_{n}(H,B)$ for $n\geq 0.$  
    
Observe that for any $\gamma:(J_{m_i}^{\Box n},\partial J_{m_i}^{\Box n},\bar{J}_{m_i}^{\Box (n-1)})\rightarrow (G,A,x_0)$ representing a class in $ \overline{\pi}_{n}(G,A)$, we have
\begin{align*} f_n^c
\circ H_n([\gamma])&= f_n^c ([\sum \limits_{0\leq i_k\leq m_k-1\atop
1\leq k\leq n}(-1)^{T(\sigma_{i_1,i_2,\cdots,i_n})}\gamma_{i_1,i_2,\cdots, i_n}]) \\
& = [\sum \limits_{0\leq i_k\leq m_k-1\atop
1\leq k\leq n}(-1)^{T(\sigma_{i_1,i_2,\cdots,i_n})} f\circ \gamma_{i_1,i_2,\cdots, i_n}]
\end{align*}
and 
\begin{align*} 
H_n\circ f_n([\gamma])&= H_n([f\circ \gamma]) \\
 &=  [\sum \limits_{0\leq i_k\leq m_k-1\atop
1\leq k\leq n}(-1)^{T(\sigma'_{i_1,i_2,\cdots,i_n})}(f\circ\gamma)_{i_1,i_2,\cdots, i_n}]. 
\end{align*}
Since $f\circ \gamma: (J_{m_i}^{\Box n},\partial J_{m_i}^{\Box n},\bar{J}_{m_i}^{\Box (n-1)})\rightarrow (H,B,y_0) $ and $\gamma:(J_{m_i}^{\Box n},\partial J_{m_i}^{\Box n},\bar{J}_{m_i}^{\Box (n-1)})\rightarrow (G,A,x_0) $ have the same domain $J_{m_i}^{\Box n}$, we have $T(\sigma'_{i_1,i_2,\cdots,i_n}) = T(\sigma_{i_1,i_2,\cdots,i_n})$. Hence $ f_n^c
\circ H_n([\gamma]) = H_n\circ f_n([\gamma])$.
\end{proof}
\end{prop}

Next, we show that the Hurewicz homomorphism $H_n: \overline{\pi}_n(G,A)\rightarrow H^c_n(G,A)$ is natural with respect to the boundary maps in the long exact sequences of relative digraph homotopy groups and cubical digraph homology groups. The boundary map for relative digraph homotopy groups is $\partial_{n+1}: \overline{\pi}_{n+1}(G,A)\rightarrow \overline{\pi}_n(A)$. By~\cite[Proposition 7]{GMJ21}, for any digraph pair $(G,A)$ there is a long exact sequence of cubical digraph homology groups
$$\cdots \rightarrow  H^c_{n+1}(G,A) \stackrel{\xi^c_{n+1}}{\rightarrow} H^c_{n}(A) \stackrel{i^c_n}{\rightarrow} H_n^c(G)\stackrel{q^c_n}{\rightarrow}  H_n^c(G,A)\rightarrow \cdots$$ where $\xi_{n+1}^c(\tau) =\sum \limits_{i=1}^n (-1)^i(\tau\circ F_{i,0}-\tau\circ F_{i,1}).$ 
\begin{prop}\label{commut}
Let $(G,A)$ be a digraph pair. Then there is a commutative diagram  $$\xymatrix{
    \overline{\pi}_{n+1}(G,A) \ar[d]_{H_{n+1}} \ar[r]^-{\partial_{n+1}} & \overline{\pi}_{n}(A) \ar[d]^{H_n}  \\
    H^c_{n+1}(G,A)  \ar[r]^-{\xi^c_{n+1}} & H^c_{n}(A).             } $$  
 \begin{proof}
Let $\gamma:(J_{m_i}^{\Box (n+1)},\partial J_{m_i}^{\Box (n+1)},\bar{J}_{m_i}^{\Box n})\rightarrow (G,A,x_0)$ be a digraph map. Since $$(J_{m_i}^{\Box (n+1)},\partial J_{m_i}^{\Box (n+1)},\bar{J}_{m_i}^{\Box n}) = (J_{m_1},\partial J_{m_1},m_1) \Box (J_{m_i},\partial J_{m_i})^{\Box n},$$ we denote a vertex in $(J_{m_i}^{\Box (n+1)},\partial J_{m_i}^{\Box (n+1)},\bar{J}_{m_i}^{\Box n})$  by $(i,a)$ for $i \in J_{m_1}$ and $a\in J_{m_i}^{\Box n}$.
By the definition of $\partial_{n+1}$ in~(\ref{jdef}), $\partial_{n+1}([\gamma]) = \widetilde{\gamma}$, where $\widetilde{\gamma}(a) = \gamma(0,a)$. Therefore 
\begin{align*} 
H_n\circ \partial_{n+1}([\gamma])= H_n([\widetilde{\gamma}]) =
  [\sum \limits_{0\leq i_k\leq m_k-1\atop
2\leq k\leq n+1}(-1)^{T(\sigma_{i_2,\cdots,i_{n+1}})}\widetilde{\gamma}_{i_2,\cdots, i_{n+1}}]. 
\end{align*} 
This equation will be used at the end of the string of equalities 

\begin{align*} 
& \xi_{n+1}^c \circ H_{n+1}([\gamma])\\
&= \xi_{n+1}^c ([\sum \limits_{0\leq i_k\leq m_k-1\atop
1\leq k\leq n+1}
(-1)^{T(\sigma_{i_1,i_2,\cdots,i_{n+1}})}\gamma_{i_1,i_2,\cdots, i_{n+1} } ]) \\
& =    \xi_{n+1}^c ([\sum \limits_{0\leq i_1\leq m_1} \sum \limits_{0\leq i_k\leq m_k-1\atop
2\leq k\leq n+1}(-1)^{T(\sigma_{i_1,i_2,\cdots,i_{n+1}})}\gamma_{i_1,i_2,\cdots, i_{n+1}}])   \\
& = [\sum \limits_{0\leq i_1\leq m_1} \sum \limits_{0\leq i_k\leq m_k-1\atop
2\leq k\leq n+1}(-1)^{T(\sigma_{i_1,i_2,\cdots,i_{n+1}})} \xi_{n+1}^c(\gamma_{i_1,i_2,\cdots, i_{n+1}})] \\
& = [\sum \limits_{0\leq i_1\leq m_1} \sum \limits_{0\leq i_k\leq m_k-1\atop
2\leq k\leq n+1}(-1)^{T(\sigma_{i_1,\cdots,i_{n+1}})} 
( \sum\limits_{j=1}^{n+1} (-1)^j(\gamma_{i_1,\cdots, i_{n+1}} \circ F_{j,0}-\gamma_{i_1,\cdots, i_{n+1}} \circ F_{j,1})]\\ 
& = [\sum \limits_{0\leq i_1\leq m_1} \sum \limits_{0\leq i_k\leq m_k-1\atop
2\leq k\leq n+1}(-1)^{T(\sigma_{i_1,\cdots,i_{n+1}})} 
( \sum\limits_{j=2}^{n+1} (-1)^j(\gamma_{i_1,\cdots, i_{n+1}} \circ F_{j,0}-\gamma_{i_1,\cdots, i_{n+1}} \circ F_{j,1}) \\
& \hspace{10mm} + \sum \limits_{0\leq i_1\leq m_1} \sum \limits_{0\leq i_k\leq m_k-1\atop
2\leq k\leq n+1}(-1)^{T(\sigma_{i_1,\cdots,i_{n+1}})}(\gamma_{i_1,\cdots, i_{n+1}} \circ F_{j,1}-\gamma_{i_1,\cdots, i_{n+1}} \circ F_{j,0})  ]  \\
& = [\sum \limits_{0\leq i_1\leq m_1} \sum \limits_{0\leq i_k\leq m_k-1\atop
2\leq k\leq n+1}\sum\limits_{j=2}^{n+1}(-1)^{T(\sigma_{i_1,\cdots,i_{n+1}})+j} 
(\gamma_{i_1,\cdots, i_{n+1}} \circ F_{j,0}-\gamma_{i_1,\cdots, i_{n+1}} \circ F_{j,1})] \\
& \hspace{10mm} +  [\sum \limits_{0\leq i_k\leq m_k-1\atop
2\leq k\leq n+1}(-1)^{T(\sigma_{i_2,\cdots,i_{n+1}})}(\sum \limits_{0\leq i_1\leq m_1} (-1)^{t(\sigma^1_{i_1,\cdots,i_{n+1}}(0), \sigma^1_{i_1,\cdots,i_{n+1}}(1))} (\gamma_{i_1,\cdots, i_{n+1}} \circ F_{j,1}-\gamma_{i_1,\cdots, i_{n+1}} \circ F_{j,0})  ] \\
& = [
\sum \limits_{0\leq i_1\leq m_1} \sum \limits_{0\leq i_k\leq m_k-1\atop 2\leq k\leq n+1, k\neq j} \sum\limits_{j=2}^{n+1}(-1)^{T(\sigma_{i_1,\cdots,i_{j-1},i_{j+1},\cdots,i_{n+1}})+j} \\
& \hspace{10mm} (\sum \limits_{0\leq i_j\leq m_j-1} (-1)^{2t(\sigma^j_{i_1,\cdots,i_{n+1}}(0), \sigma^j_{i_1,\cdots,i_{n+1}}(1))} (\gamma_{i_1,\cdots, i_{n+1}}|_{J_1\Box \cdots\{i_j\}\Box\cdots J_1}-\gamma_{i_1,\cdots, i_{n+1}}|_{J_1\Box \cdots\{i_j+1\}\Box\cdots J_1}) )]\\
& \hspace{10mm} + [ \sum \limits_{0\leq i_k\leq m_k-1\atop
2\leq k\leq n+1}(-1)^{T(\sigma_{i_2,\cdots,i_{n+1}})}\\
& \hspace{10mm}(\sum \limits_{0\leq i_1\leq m_1} (-1)^{2t(\sigma^1_{i_1,\cdots,i_{n+1}}(0), \sigma^1_{i_1,\cdots,i_{n+1}}(1))} (\gamma_{\{0\},\cdots, i_{n+1}}|_{0\Box J_1^{n}}-\gamma_{m_1-1,\cdots, i_{n+1}}|_{\{m_1\}\Box J_1^{n}})  ]  
 \end{align*}
 \begin{align*}
& = [
\sum \limits_{0\leq i_1\leq m_1} \sum \limits_{0\leq i_k\leq m_k-1\atop 2\leq k\leq n+1, k\neq j} \sum\limits_{j=2}^{n+1}(-1)^{T(\sigma_{i_1,\cdots,i_{j-1},i_{j+1},\cdots,i_{n+1}})+j}  \\
 & \hspace{10mm} (\gamma_{i_1,\cdots,\{0\},i_{j+1},\cdots, i_{n+1}}|_{J_1\Box\cdots\{0\}\cdots J_1}-\gamma_{i_1,\cdots,\{m_j-1\},\cdots, i_{n+1}}|_{J_1\Box\cdots\{m_j\}\cdots J_1})]\\
& \hspace{10mm} + [ \sum \limits_{0\leq i_k\leq m_k-1\atop
2\leq k\leq n+1}(-1)^{T(\sigma_{i_2,\cdots,i_{n+1}})}
 (\gamma_{\{0\},\cdots, i_{n+1}}|_{0\Box J_1^{n}}-\gamma_{m_1-1,\cdots, i_{n+1}}|_{\{m_1\}\Box J_1^{n}})  ]  \\
& =[0]+  [ \sum \limits_{0\leq i_k\leq m_k-1\atop
2\leq k\leq n+1}(-1)^{T(\sigma_{i_2,\cdots,i_{n+1}})}
 \gamma_{\{0\},\cdots, i_{n+1}}|_{0\Box J_1^{n}}  ] \\
 & =  [ \sum \limits_{0\leq i_k\leq m_k-1\atop
2\leq k\leq n+1}(-1)^{T(\sigma_{i_2,\cdots,i_{n+1}})}
 \widetilde{\gamma}_{i_2,\cdots, i_{n+1}}  ] \\
 & = H_n\circ \partial_{n+1}([\gamma]).
\end{align*}
This proves the proposition. 
    \end{proof}
\end{prop}

\subsection{The digraph suspension homomorphism} We begin with a definition.
\begin{defn}\cite[Definition 5.9]{GLMY20}
The \emph{cone} of a digraph \( X \) is a digraph $CX$ whose vertex set is that of $X$ together with one additional vertex $a$, and whose arrow set consists of those for $X$ together with additional arrows of the form \( b \to a \) for every \( b \in X \). If $X$ is a based digraph with base-point $x_{0}$ then $CX$ is also based with base-point $x_{0}$. The vertex \( a \) is called the cone vertex.
\end{defn}
\begin{example}\label{cone}Let $X$ be a square. The cone $CX$ of $X$ is shown as follows:

\begin{center}
\begin{tikzpicture}[scale =0.8]
  \begin{scope}[shift={(-1,0)}] 

  \draw[->, thick] (-4,0) -- (-3.5,0.95);
    \draw[->, thick] (-4,0) -- (-1.05,0);
    \draw[->, thick] (-3.5,1) -- (-0.55,1);
    \draw[->, thick] (-0.5,0.95) --  (-1,0);

    \node[above, black] at (-3.5,1) {1};
    \node[above, black] at (-0.5,1) {2};
    \node[below, black] at (-1,0) {3};
    \node[below, black] at (-4,0) {0};
    \node[above] at (-2.5,-1) {$X$};
    \filldraw (-1,0) circle (.04)
(-3.5,1) circle (.04)
(-4 ,0) circle (.04)
(-0.5,1) circle (.04);
 \end{scope}

\begin{scope}[shift={(1,0)}] 
  \node[above, red] at (3,2) {$a$};
    \node[above, black] at (1.5,1) {1};
     \node[above, black] at (4.5,1) {2};
     \node[below, black] at (4,0) {3};
   \node[below, black] at (1,0) {0};
  \node[above] at (2.7,-1.5) {$CX$};

   \draw[->, thick] (1,0) -- (3.9,0);
   \draw[->, thick] (1,0) -- (1.4,0.9);
    \draw[->, thick] (1.5,1) -- (4.4,1);
    \draw[->,thick] (4.4,0.9) -- (4,0);
  \draw[->,red, thick] (1,0) -- (2.95,1.95);
   \draw[->, red,thick] (4,0) -- (3.05,1.95);
    \draw[->, red,thick] (1.5,1) -- (2.95,1.95);
    \draw[->, red, thick] (4.5,1) -- (3.05,1.95);

\filldraw (1,0) circle (.04)
(3,2) circle (.04)
(4 ,0) circle (.04)
(4.5,1) circle (.04)
(1.5,1) circle (.04);
\end{scope}
\end{tikzpicture}
\end{center}
\end{example}
\begin{defn}\cite[Definition 6.13]{GLMY20}
The \emph{suspension} of a digraph \( X \) is a digraph $\Sigma X$ whose vertex set is that of $X$ together with two additional vertices $a$ and $b$, and whose arrow set consists of those for $X$ together with additional arrows of the form \( c \to a \) and \( c \to b \) for every \( c \in X \).  If $X$ is a based digraph with base-point $x_{0}$ then $\Sigma X$ is also based with base-point $x_{0}$. The vertices \( a \) and \( b \) are called the suspension vertices.
\end{defn}
\begin{example}\label{susp}Let $X$ be a square. The suspension $\Sigma X$ of $X$ is shown as follows:

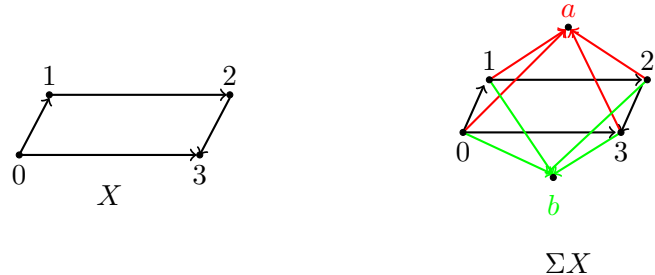
\begin{center}
\begin{tikzpicture}
    \begin{scope}[shift={(-1,-0.3)}] 
        \draw[->, thick] (-3.2,0) -- (-2.8,0.76);
        \draw[->, thick] (-3.2,0) -- (-0.84,0);
        \draw[->, thick] (-2.8,0.8) -- (-0.44,0.8);
        \draw[->, thick] (-0.4,0.76) -- (-0.8,0);

        \node[above, black] at (-2.8,0.8) {1};
        \node[above, black] at (-0.4,0.8) {2};
        \node[below, black] at (-0.8,0) {3};
        \node[below, black] at (-3.2,0) {0};
        \node[above] at (-2,-0.8) {$X$};

        \filldraw (-0.8,0) circle (.04)
        (-2.8,0.8) circle (.04)
        (-3.2 ,0) circle (.04)
        (-0.4,0.8) circle (.04);
    \end{scope}

    \begin{scope}[shift={(1,0)}] 
        \node[below, green] at (1.9,-0.7) {$b$};
        \node[above, red] at (2.1,1.4) {$a$};
        \node[above, black] at (1.05,0.7) {1};
        \node[above, black] at (3.15,0.7) {2};
        \node[below, black] at (2.8,0) {3};
        \node[below, black] at (0.7,0) {0};
        \node[above] at (2.1,-2) {$\Sigma X$};

        \draw[->, thick] (0.7,0) -- (2.73,0);
        \draw[->, thick] (0.7,0) -- (0.98,0.63);
        \draw[->, thick] (1.05,0.7) -- (3.08,0.7);
        \draw[->,thick] (3.08,0.63) -- (2.8,0);

        \draw[->,red, thick] (0.7,0) -- (2.07,1.37);
        \draw[->, red,thick] (2.8,0) -- (2.13,1.37);
        \draw[->, red,thick] (1.05,0.7) -- (2.07,1.37);
        \draw[->, red, thick] (3.15,0.7) -- (2.13,1.37);
        \draw[->,green, thick] (0.7,0) -- (1.9,-0.56);
        \draw[->, green,thick] (2.8,0) -- (1.9,-0.56);
        \draw[->, green,thick] (1.05,0.7) -- (1.9,-0.56);
        \draw[->, green, thick] (3.15,0.7) -- (1.8,-0.56);

        \filldraw (0.7,0) circle (.04)
        (2.1,1.4) circle (.04)
        (2.8 ,0) circle (.04)
        (3.15,0.7) circle (.04)
        (1.05,0.7) circle (.04)
        (1.9,-0.6) circle (.04);
    \end{scope}
\end{tikzpicture}
\end{center}
\end{example}
Observe that, as in the case of topological spaces, the suspension of $X$ is the union of two cones over $X$. Explicitly, $$\Sigma X= C^{+}X\cup_{X} C^{-}X$$ where $C^{+}X$ has cone vertex $a$ and $C^{-}X$ has cone vertex $b$.

Now observe that there is a commutative diagram of based digraph inclusions
 $$\xymatrix{
    X \ar[d] \ar[r] & C^{+}X \ar[d]  \\
    C^{-}X  \ar[r] & \Sigma X.             } $$
This implies that there is a map of digraph pairs
$f:(C^{+}X,X)\rightarrow (\Sigma X,C^{-}X)$. Proposition~\ref{mor} then implies that there is a commutative diagram of exact sequences
$$\xymatrix@C=0.5cm{
  \overline{\pi}_{n+1}(X) \ar[r]\ar[d] &  \overline{\pi}_{n+1}(C^{+}X) \ar[r]\ar[d] &
    \overline{\pi}_{n+1}(C^{+}X,X)\ar[r]^-{\partial_{n+1}}\ar[d]^{f_{n+1}}&
  \overline{\pi}_n(X) \ar[r]\ar[d]  & \overline{\pi}_n(C^{+}X)\ar[d] \\
   \overline{\pi}_{n+1}(C^{-}X) \ar[r] &  \overline{\pi}_{n+1}(\Sigma X) \ar[r]^-{j_{n+1}}&
    \overline{\pi}_{n+1}(\Sigma X,C^{-}X)\ar[r]&
  \overline{\pi}_n(C^{-}X) \ar[r]  & \overline{\pi}_n(\Sigma X)
  }$$
which is exact as sets for $n\geq 0$ and as groups for $n\geq 1$. By~\cite[Example~3.11]{GLMY15}, the cone $CX$ is contractible. Thus both $C^{+}X$ and $C^{-}X$ are contractible. By~\cite[Corollary 4.11]{LWYZ24}, a contractible digraph $Y$ has the property that $\overline{\pi}_{n}(Y)\cong\overline{\pi}_{n}(\ast)$ for a one-vertex digraph $\ast$, all of whose digraph homotopy groups are zero, and hence $\overline{\pi}_{n}(Y)\cong 0$ for all $n\geq 0$. Thus in the diagram of exact sequences above, both $j_{n+1}$ and $\partial_{n+1}$ are isomorphisms.

\begin{defn} For $n\geq 0$, define the \emph{suspension map} by the composite
$$E_n:\overline{\pi}_{n}(X)\stackrel{\partial_{n+1}^{-1}}{\longrightarrow}\overline{\pi}_{n+1}(C^{+}X,X)\stackrel{f_{n+1}}{\longrightarrow}\overline{\pi}_{n+1}(\Sigma X,C^{-}X)\stackrel{j_{n+1}^{-1}}{\longrightarrow}\overline{\pi}_{n+1}(\Sigma X).$$ 
\end{defn}
Observe that if $n\geq 1$ then $E_n$ is a group homomorphism since each of $j_{n+1}^{-1}$, $f_{n+1}$ and $\partial^{-1}_{n+1}$ are. Observe also that $E_n$ need not be an isomorphism if $f_{n+1}$ is not. 

By~(\ref{jdef}), the map $\overline{\pi}_{n+1}(\Sigma X)\stackrel{j_{n+1}}{\longrightarrow}\overline{\pi}_{n}(\Sigma X,C^{-}X)$ is defined as the composite 
$$j_{n+1}:\overline{\pi}_{n+1}(\Sigma X)\stackrel{\delta_{n}}{\longrightarrow}\overline{\pi}_{n}(P_{i})\stackrel{\theta_{n}}{\longrightarrow}\overline{\pi}_{n}(\overline{P}(\Sigma X,C^{-}X,x_{0}))\stackrel{\Phi_{n+1}^{-1}}{\longrightarrow}\overline{\pi}_{n+1}(\Sigma X,C^{-}X),$$ 
where $P_{i}$ is the mapping path digraph of the inclusion 
$C^{-}X\stackrel{i}{\rightarrow}\Sigma X$, 
so it is not obviously induced by a map of digraph pairs. But it will be useful to show that, in fact, it is.  

\begin{lem}\label{induce1} 
The map 
$\overline{\pi}_{n+1}(\Sigma X)\stackrel{j_{n+1}}{\longrightarrow}\overline{\pi}_{n}(\Sigma X,C^{-}X)$ 
is induced by the map of pairs 
$(\Sigma X,x_{0})\rightarrow (\Sigma X,C^{-}X)$. 
\end{lem} 

\begin{proof} 
The commutative diagram of based digraph inclusions 
$$\xymatrix{
    x_{0} \ar[d] \ar[r] & \Sigma X \ar[d]  \\
    C^{-}X  \ar[r] & \Sigma X.        } $$ 
implies that there is a map of digraph pairs $(\Sigma X,x_{0})\rightarrow (\Sigma X,C^{-}X)$. Proposition~\ref{mor} then implies that there is a commutative diagram of exact sequences
$$\xymatrix@C=0.5cm{
  \overline{\pi}_{n+1}(x_{0}) \ar[r]\ar[d] &  \overline{\pi}_{n+1}(\Sigma X) \ar[r]\ar[d]^{=} &
    \overline{\pi}_{n+1}(\Sigma X,x_{0})\ar[r]^-{\partial_{n+1}}\ar[d]^{g_{n+1}}&
  \overline{\pi}_n(x_{0}) \ar[r]\ar[d]  & \overline{\pi}_n(\Sigma X)\ar[d]^{=} \\
   \overline{\pi}_{n+1}(C^{-}X) \ar[r] &  \overline{\pi}_{n+1}(\Sigma X) \ar[r]^-{j_{n+1}}&
    \overline{\pi}_{n+1}(\Sigma X,C^{-}X)\ar[r]&
  \overline{\pi}_n(C^{-}X) \ar[r]  & \overline{\pi}_n(\Sigma X)
  }$$
which is exact as sets for $n\geq 0$ and as groups for $n\geq 1$. 
By definition of $\overline{\pi}_{n}(G)$, the one vertex digraph $\{x_{0}\}$ has the property that $\overline{\pi}_{n}(x_{0})\cong 0$ for all $n\geq 0$. Thus the map 
$\overline{\pi}_{n}(\Sigma X)\rightarrow\overline{\pi}_{n+1}(\Sigma X,x_{0})$ 
in the top row is an isomorphism (in fact, it is the identity map). The commutativity of the middle square then implies that $j_{n+1}=g_{n+1}$. Hence $j_{n+1}$ is induced by a map of digraph pairs. 
\end{proof}

\subsection{Relating the digraph Hurewicz and suspension homomorphisms} 
We will relate the homotopy suspension $E^n$ to homological analogues of the suspension in both cubical and GLMY homology. To do so, observe that if $(G,A)$ is a digraph pair then there is a short exact sequence of 
cubical chain groups $$ 0\rightarrow C^c_n(A)\rightarrow C^c_n(G)\rightarrow C^c_n(G)/ C_n(A)\rightarrow 0.$$ This induces a long exact sequence of homology groups 
$$\cdots \rightarrow  H^c_{n+1}(G,A) \stackrel{\xi^c_{n+1}}{\rightarrow} H^c_{n}(A) \stackrel{i_n^c}{\rightarrow} H^c_n(G)\stackrel{q_n^c}{\rightarrow}  H^c_n(G,A)\rightarrow \cdots.$$ A similar short exact sequence of path chain groups induces a long exact sequence of GLMY homology groups
$$\cdots \rightarrow  H_{n+1}(G,A) \stackrel{\xi_{n+1}}{\rightarrow} H_{n}(A) \stackrel{i_n}{\rightarrow} H_n(G)\stackrel{q_n}{\rightarrow}  H_n(G,A)\rightarrow \cdots.$$
 A homology suspension can be defined in each case.
\begin{defn}
Let $X$ be a digraph.  For $n\geq 0$ the \emph{cubical homology suspension} $E^c_n$ is the composite of homomorphisms $$E^c_n: H_n^c (X) \stackrel{(\xi^c_{n+1})^{-1}}{\longrightarrow} H_{n+1}^c(C^+ X,X) \stackrel{i^c_{n+1}}{\longrightarrow} H_{n+1}^c(\Sigma X,C^{-}X ) \stackrel{(q_{n+1}^c)^{-1}}{\longrightarrow}  H_{n+1}^c(\Sigma X),$$ 
and 
the \emph{GLMY homology suspension} $E'_n$ is the composite of homomorphisms $$E'_n: H_n (X) \stackrel{\xi_{n+1}^{-1}}{\longrightarrow} H_{n+1}(C^+ X,X) \stackrel{i_{n+1}}{\longrightarrow} H_{n+1}(\Sigma X,C^{-}X ) \stackrel{q_{n+1}^{-1}}{\longrightarrow}  H_{n+1}(\Sigma X).$$
\end{defn}
As is well known, for topological spaces the homology suspension is an isomorphism. We show this also holds for  the GLMY homology suspension $E'_n$. 
\begin{prop}\label{susiso}Let $X$ be a digraph. The GLMY homology suspension $$E'_n:H_n (X) \rightarrow H_{n+1} (\Sigma X)  $$ is an isomorphism for $n\geq 0$.
\begin{proof}
To prove that $E'_n:H_n (X) \rightarrow H_{n+1} (\Sigma X)  $ is an isomorphism, as $\xi_{n+1}$ and $q_{n+1}$ are isomorphisms, it suffices to show that $i_{n+1}: H_{n+1}(C^{+}X,X)\rightarrow H_{n+1}(\Sigma X,C^-X)$ is an isomorphism.

By~\cite[Propositions 6.10 and 6.14]{GLMY12}, $$ i: \Omega_{n+1}(C^{+}X) /\Omega_{n+1}(X)\rightarrow   \Omega_{n+1}(\Sigma X)/ \Omega_{n+1}(C^{-}X)$$ is defined by $$ i([v_0v_1\cdots v_nv_{n+1}])=\left\{
                     \begin{array}{ll}
                      [v_0v_1\cdots v_na], & \hbox{$v_{n+1}=a$} \\
                       0, & \hbox{$v_{n+1}\in X$}
                     \end{array}
                   \right
.$$ 
for any allowed path $v_0v_1\cdots v_nv_{n+1}$ in $C^{+}X$.

Since $\Omega_{n+1}(\Sigma X) = \Omega_n(X)
\bigotimes span\{a,b\}$ and $\Omega_{n+1}(C^- X) = \Omega_n(X)
\bigotimes span\{b\}$, a non-zero element in $ \Omega_{n+1}(\Sigma X)/ \Omega_{n+1}(C^- X) $ must be a linear combination of terms $[v_0v_1\cdots v_n a]$. Each $v_0v_1\cdots v_n a$ is in $C^{+}X$ but not in $X$. Thus $[v_0v_1\cdots v_n a ]$ is in  $ \Omega_{n+1}(C^{+} X)/ \Omega_{n+1}(X) $. Consequently, $i$ is surjective.

Assume that $i([\sum v_0v_1\cdots v_{n+1}]) =[0]$ for some element $[\sum v_0v_1\cdots v_{n+1}]\in \Omega_{n+1}(C^+ X)/ \Omega_{n+1}(X)$.  Then $[\sum v_0v_1\cdots v_{n+1}]\in \Omega_{n+1}(C^+ X)$ implies that $\sum v_0v_1\cdots v_{n+1} \in A_{n+1}(C^+ X)$, while the assumption that $i([\sum v_0v_1\cdots v_{n+1}]) =[0]$ implies that $[\sum v_0v_1\cdots v_{n+1}]\in \Omega_{n+1}(C^- X)$ and therefore $\sum v_0v_1\cdots v_{n+1} \in A_{n+1}(C^- X)$. Thus $$\sum v_0v_1\cdots v_{n+1} \in A_{n+1}(C^- X) \bigcap A_{n+1}(C^+ X)= A_{n+1}(X).$$ Consequently, $[\sum v_0v_1\cdots v_{n+1}] =0 \in \Omega_{n+1}(C^+ X)/ \Omega_{n+1}(X)$, implying that $i$ is injective. Hence $i_{n+1}$ is an isomorphism. 
\end{proof}
\end{prop} 
Next, we show that the homotopy and cubical homology suspensions are compatible with the Hurewicz homomorphism.
\begin{prop}\label{ccompa}
For each $n\geq 0$ there is a commutative diagram 
$$\xymatrix{
     \overline{\pi}_{n}(X)\ar[r]^-{E_n}\ar[d]^{H_n} & \overline{\pi}_{n+1}(\Sigma X)\ar[d]^{H_{n+1}} \\
     H_{n}^c(X)\ar[r]^-{E_n^c} & H^c_{n+1}(\Sigma X).}$$
\begin{proof} 
Consider the following diagram
$$\xymatrix{
\overline{\pi}_n (X) \ar[r]^-{\partial_{n}^{-1}}\ar[d]_{H_n}& \overline{\pi}_{n+1}(C^+ X,X) \ar[r]^{i_{n+1}} \ar[d]^{H_{n+1}} & \overline{\pi}_{n+1}(\Sigma X,C^{-}X ) \ar[d]^{H_{n+1}} \ar[r]^-{j_{n+1}^{-1}} & \overline{\pi}_{n+1}(\Sigma X) \ar[d]^{H_{n+1}}  \\
H_n^c (X) \ar[r]^-{(\xi^c_{n+1})^{-1}}& H_{n+1}^c(C^+ X,X) \ar[r]^-{i^c_{n+1}}& H_{n+1}^c(\Sigma X,C^{-}X ) \ar[r]^-{(q_{n+1}^c)^{-1}} & H_{n+1}^c(\Sigma X) . } $$ 
The top row is the definition of the homotopy suspension $E_n$ and the bottom row is the definition of the cubical homology suspension $E^{c}_n$.
By Proposition~\ref{commut}, $\xi^c_{n+1}\circ H_{n+1} = H_{n}\circ \partial_{n+1}$, so composing on the left with $(\xi^c_{n+1})^{-1}$ and on the right with $\partial_{n+1}^{-1}$ implies that $H_{n+1}\circ \partial_{n+1}^{-1} = (\xi_{n+1} )^{-1}\circ H_{n}$
and therefore the left square commutes. 
By Proposition~\ref{cubnat}, $H_{n+1}\circ i_{n+1} = i_{n+1}^c\circ H_{n+1}$, so the middle square commutes. By Lemma \ref{induce1}, $j_{n+1}$ is induced by the map of pairs $l:(\Sigma X,x_{0})\rightarrow (\Sigma X,C^-X)$. We claim that $q^c_{n+1}$ is also induced by $l:(\Sigma X,x_{0})\rightarrow (\Sigma X,C^-X)$. If so, then Proposition \ref{cubnat} implies that $q^c_{n+1}\circ H_{n+1} = H_{n+1}\circ j_{n+1}$, so composing on the right with $(q^c_{n+1})^{-1}$ and on the right with $(j_{n+1})^{-1}$ gives $H_{n+1}\circ E_{n}=E_{n}^{c}\circ H_{n}$. Thus the right square commutes, implying the entire diagram does. From the commutativity of the entire diagram we obtain $H_{n+1}\circ E_{n}=E^{c}_{n}\circ H_{n}$, proving the proposition. 

It remains to show that $q^c_{n+1}$ is also induced by the map $l:(\Sigma X,x_{0})\rightarrow (\Sigma X,C^-X)$.  Since $l$ induces a homomorphism between long exact sequences of cubical homology groups, there is a commutative diagram  $$\xymatrix@C=0.5cm{
  H_{n+1}^c(x_{0}) \ar[r]\ar[d] &  H_{n+1}^c(\Sigma X) \ar[r]^{=}\ar[d]^{=} &
    H_{n+1}^c(\Sigma X,x_{0})\ar[r]^{}\ar[d]^{l_{n+1}}&
  H_n^c(x_{0}) \ar[r]\ar[d]  & H_n^c(\Sigma X)\ar[d]^{=} \\
  H_{n+1}^c(C^{-}X) \ar[r] &  H_{n+1}^c(\Sigma X) \ar[r]^-{q^c_{n+1}}&
    H_{n+1}^c(\Sigma X,C^{-}X)\ar[r]&
 H_n^c(C^{-}X) \ar[r]  & H_n^c(\Sigma X).
  }$$
Thus $q^c_{n+1} =l_{n+1}$, that is, $q^c_{n+1}$ is induced by $l:(\Sigma X,x_{0})\rightarrow (\Sigma X,C^-X)$. 
\end{proof}
\end{prop}

Now we relate the suspension maps for cubical and GLMY homology. 
Recall from~(\ref{Lndef}) that for a digraph $G$ there is a chain map $\iota_{n}:\Omega^c_n(G)\rightarrow\Omega_n(G)$ that induces a homomorphism $L_n: H_n^c(G)\rightarrow H_n(G)$. 

\begin{prop}\label{hhcom}
Let $X$ be a digraph. For each $n\geq 0$ there is a commutative diagram 
$$\xymatrix{
H_{n}^c(X) \ar[r]^-{E^c_{n}} \ar[d]_{L_{n}} & H_{n+1}^c(\Sigma X) \ar[d]^{L_{n+1}} \\
H_{n}(X) \ar[r]^-{E'_{n}}  &     H_{n+1}(\Sigma X).  } $$
\begin{proof}
First observe that the inclusion $i: X \rightarrow C^+ X$ induces a short exact sequence of cubical chain groups 
$$\xymatrix{
0 \ar[r]&   \Omega_{n}^c(X) \ar[r]^{i^c_n} & \Omega_{n}^c(C^+ X) \ar[r]^{q^c_n}  & \Omega_{n}^c(C^{+}X,X) \ar[r]&0      } $$ 
and a short exact sequence of GLMY chain groups
$$\xymatrix{
0 \ar[r]&   \Omega_{n}(X) \ar[r]^{i_n} & \Omega_{n}(C^+ X) \ar[r]^{q_n}  & \Omega_{n}(C^{+}X,X) \ar[r]&0      .} $$  
By~\cite[Proposition 10]{GMJ21}, the inclusion $i: X\rightarrow C^+X$ induces a commutative diagram  
$$\xymatrix{
0 \ar[r]&   \Omega_{n}^c(X) \ar[r]^{i^c_n}\ar[d]_{\iota_n} & \Omega_{n}^c(C^+ X) \ar[r]^{q^c_n} \ar[d]_{\iota_n} & \Omega_{n}^c(C^{+}X,X) \ar[r] \ar[d]_{\iota_n} &0     \\
0 \ar[r] & \Omega_{n}(X) \ar[r]^{i_n} &\Omega_{n}(C^+ X) \ar[r]^{q_n} & \Omega_{n}(C^{+}X,X)\ar[r] & 0.} $$ 
Hence, taking homology, there is a commutative diagram of long exact sequences 
 $$\xymatrix{
\cdots \ar[r]&H_{n+1}^c(C^+ X,X) \ar[r]^-{\xi^c_{n+1}} \ar@{-->}[d]_{L_{n+1}} & H_{n}^c(X) \ar[d]^{L_n} \ar[r]^{i^c_n} & H_{n}^c(C^+ X) \ar[d]^{L_n} \ar[r]^{q_n^c} &  H_{n}^c(C^+ X,X) \ar@{-->}[d]^{L_n} \ar[r] & \cdots \\
\cdots \ar[r]&H_{n+1}(C^+ X,X) \ar[r]^-{\xi_{n+1}}  &     H_{n}(X)  \ar[r]^{i_n} & H_{n}(C^+ X) \ar[r]^{q_n}  &   H_{n}(C^+ X,X) \ar[r]& \cdots         } $$ 
that defines the maps $L_n$ and $L_{n+1}$ of relative homology groups. In particular, from the left square we obtain $L_{n+1}\circ \xi^c_{n+1} = \xi_{n+1} \circ L_{n+1}.$ As $\xi^{c}_{n+1}$ and $\xi_{n+1}$ are isomorphisms, composing on the right with $(\xi^c_{n+1})^{-1}$ and on the left with $(\xi_{n+1} )^{-1}$ implies that (i) $(\xi_{n+1} )^{-1}\circ L_{n+1} = L_{n+1} \circ (\xi^c_{n+1})^{-1} .$ Similarly, from the right square above we obtain (ii) $ q_{n+1}^{-1}\circ L_{n+1} = L_{n+1} \circ (q^c_{n+1})^{-1} . $ 

Next, the sequences of inclusions $X\rightarrow C^+X\rightarrow\Sigma X$ and $X\rightarrow C^-X\rightarrow\Sigma X$ induce the commutative cube of chain maps on the left side of the diagram 
 $$
 \begin{small}
 \xymatrix@R=0.2em@C=0.05em{
&  \Omega_{n+1}^c (C^{-}X) \ar[rr]\ar[dd]_(0.7){\iota_{n+1}}& & \Omega_{n+1}^c(\Sigma X) \ar[rr]\ar[dd]_(0.7){\iota_{n+1}} & & \textcolor{red}{\Omega_{n+1}^c(\Sigma X,C^{-}X )} \ar@{-->}[dd]  \\
\Omega_{n+1}^c (X) \ar[ur] \ar[rr]\ar[dd]_{\iota_{n+1}}& & \Omega_{n+1}^c (C^{+} X) \ar[ur] \ar[rr] \ar[dd]_(0.7){\iota_{n+1}}& & \textcolor{red}{\Omega_{n+1}^c (C^{+} X,X)} \ar@{-->}[ur] \ar@{-->}[dd]&  \\
&  \Omega_{n+1} (C^{-}X) \ar[rr] & & \Omega_{n+1}(\Sigma X)  \ar[rr] & &  \textcolor{red}{\Omega_{n+1}(\Sigma X,C^{-}X )} \\
\Omega_{n+1} (X)\ar[ur] \ar[rr] & & \Omega_{n+1} (C^{+} X) \ar[ur] \ar[rr]& & \textcolor{red}{\Omega_{n+1} (C^{+} X,X). }\ar@{-->}[ur] &  }
\end{small}$$ 
 The cube on the right side of the diagram is obtained by taking short exact sequences of chain complexes horizontally, and noting that short exactness implies that the face in red commutes. Taking homology with respect to the face in red gives a commutative diagram 
$$\xymatrix{
 H_{n+1}^c(C^+ X,X) \ar[r]^{i^c_{n+1}} \ar[d]^{L_{n+1}} & H_{n+1}^c(\Sigma X,C^{-}X ) \ar[d]^{L_{n+1}}  \\
 H_{n+1}(C^+ X,X) \ar[r]^{i_{n+1}}  &     H_{n+1}(\Sigma X, C^-X). } $$ 
Combining this square with (i) and (ii) gives a commutative diagram
$$\xymatrix{
H_n^c (X) \ar[r]^-{(\xi^c_{n+1})^{-1}}\ar[d]_{L_n}& H_{n+1}^c(C^+ X,X) \ar[r]^{i^c_{n+1}} \ar[d]^{L_{n+1}} & H_{n+1}^c(\Sigma X,C^{-}X ) \ar[d]^{L_{n+1}} \ar[r]^-{(q^c_{n+1})^{-1}} & H_{n+1}^c(\Sigma X) \ar[d]^{L_{n+1}}  \\
H_n (X) \ar[r]^-{\xi_{n+1}^{-1}}& H_{n+1}(C^+ X,X) \ar[r]^{i_{n+1}}  &     H_{n+1}(\Sigma X, C^-X) \ar[r]^-{q_{n+1}^{-1}} &  H_{n+1}(X). } $$ 
Observe that the top row of this diagram is the definition of $E^c_n$ while the bottom row is the definition of $E'_n$. Thus the outer rectangle gives the commutative diagram asserted by the proposition.
\end{proof}
\end{prop} 

For a digraph $X$, define the \emph{GLMY Hurewicz homomorphism} by the composite 
$$\widetilde{H}_n:\overline{\pi}_n(X)\xrightarrow{H_n} H^c_n(X)\xrightarrow{L_n} H_n(X).$$
Combining Propositions \ref{ccompa} and \ref{hhcom} immediately implies Theorem~\ref{intropcompa}, restated as the following.
\begin{thm}\label{pcompa}
For any digraph $X$, the homotopy suspension and GLMY homology suspension are compatible, i.e. there is a commutative square $$\xymatrix{
     \overline{\pi}_{n}(X)\ar[r]^{E_n}\ar[d]_{\widetilde{H}_n} & \overline{\pi}_{n+1}(\Sigma X)\ar[d]^{\widetilde{H}_{n+1}} \\
     H_{n}(X)\ar[r]^{E'_n} & H_{n+1}(\Sigma X). }$$ 
\end{thm} 
\vspace{-0.7cm}~$\qqed$\bigskip 

In the $n=1$ case, a Hurewicz homomorphism $h':\pi_{1}(X)\rightarrow H_{1}(X)$ was defined in~\cite[Theorem~4.23]{GLMY15}, while by~\cite[Corollary 4.24]{LWYZ24} there is an isomorphism $\overline{\pi}_1(X)\cong\pi_1(X)$. We compare $h'':\overline{\pi}(X)\xrightarrow{\cong}\pi_{1}(X)\xrightarrow{h'} H_{1}(X)$ and $\widetilde{H}_1: \overline{\pi}_1(X)\rightarrow H_1(X)$ to show that they have the same image.

\begin{prop}\label{GLMYHurewicz}
For any digraph $X$, the GLMY Hurewicz homomorphism $\widetilde{H}_1: \overline{\pi}_1(X) \rightarrow H_1(X)$ has image equal to that of $h'':\overline{\pi}_1(X) \rightarrow H_1(X)$. 
\begin{proof} 
Since there is an isomorphism $\overline{\pi}_{1}(X)\cong\pi_{1}(X)$, it suffices to show that $\widetilde{H}_{1}$ and $h'$ have the same image.

Recall that we always assume the length of the $i^{th}-$coordinate $m_i$ of a grid digraph map $f:(J_{m_i}^{\Box (n+1)},$ $\partial J_{m_i}^{\Box (n+1)},\bar{J}_{m_i}^{\Box n})\rightarrow (G,A,x_0) $ is even and $m_i\geq 2$. By the definition of $h'$ in~\cite{GLMY15}, for any loop $\gamma: (J_n,\partial J_n)\rightarrow (X,x_0)$ then 
$$h'([\gamma]) =\left\{
                     \begin{array}{ll}
                      [\sum\limits_{i\rightarrow i+1} \gamma(i) \gamma(i+1) - \sum\limits_{i+1\rightarrow i} \gamma(i+1)\gamma(i)], & \hbox{$n\leq 3$} \\
                       0, & \hbox{$n= 2$.}
                     \end{array}
                   \right
.$$ 

On the other hand, by definition $\widetilde{H}_1 = H_1 \circ L_{1}$ where $L_1$ is induced by~$\iota_1$. By definition of $\iota_1$ in Subsection~\ref{cubical}, we have $\iota_1(f)=f(w_1)$ for any 1-dimensional singular chain $f: J_1\rightarrow G$. Here, for the generator $w_1$, since $P(0,1) = \{(01)\}$ we have $w_1 = 01$. Hence $\iota_1(f)=f(w_1) =f(01)= f(0)f(1)$. Now using $\widetilde{H}_1 = L_1\circ H_1$, the definitions of $L_{1}$ and $H_{1}$ and the fact that $L_{1}$ is a homomorphsm, for any loop $\gamma: (J_n,\partial J_n)\rightarrow (X,x_0)$ with $n\geq 2$ we obtain: 
\begin{align*}
\widetilde{H}_1([\gamma])& = L_1([\sum\limits_{0\leq i_1\leq m_1-1}(-1)^{T(\sigma_{i_1})} \gamma_{i_1}])\\
& = L_1([\sum\limits_{0\leq i_1\leq m_1-1}(-1)^{t(\sigma_{i_1}(0),\sigma_{i_1}(1))} \gamma_{i_1}]) \\
& = [\sum\limits_{0\leq i_1\leq m_1-1}(-1)^{t(\sigma_{i_1}(0),\sigma_{i_1}(1))} L_1(\gamma_{i_1})] \\
& = [\sum\limits_{0\leq i_1\leq m_1-1}(-1)^{t(\sigma_{i_1}(0),\sigma_{i_1}(1))} \gamma_{i_1}(J_1)] \\
& =  [\sum\limits_{0\leq i_1\leq m_1-1}(-1)^{t(\sigma_{i_1}(0),\sigma_{i_1}(1))} \gamma_{i_1}(\sigma_{i_1}(0)\sigma_{i_1}(1))] \\
& = [\sum\limits_{0\leq i_1\leq m_1-1\atop i_1\rightarrow i_1+1} \gamma(i_1)\gamma(i_1+1) - \sum\limits_{0\leq i_1\leq m_1-1\atop i_1+1\rightarrow i_1}\gamma(i_1+1)\gamma(i_1)].
\end{align*}
Thus if $n\geq 3$, we immediately obtain $\widetilde{H}_1([\gamma]) = h'([\gamma])$. If $n =2$, $\widetilde{H}_1([\gamma]) =\gamma(0)\gamma(1) -\gamma(2)\gamma(1) =0$. Hence, $\widetilde{H}_1$ and $h'$ have the same image. 
\end{proof}
\end{prop}

Since $H_1(X)$ is abelian, the Hurewicz homomorphism  
$\overline{\pi}_1(X)\xrightarrow{\widetilde{H}_1} H_1(X)$ 
factors through the abelianization $\overline{\pi}_1(X)_{ab}$ of $\overline{\pi}_1(X)$ to give a homomorphism 
$(\widetilde{H}_1)_{ab}:\overline{\pi}_{1}(X)_{ab}\rightarrow H_1(X)$. 
\begin{cor}\label{pi1abelian}  
For any digraph $X$, the map 
$\overline{\pi}_{1}(X)_{ab}\xrightarrow{(\widetilde{H}_1)_{ab}} H_1(X)$ 
is an isomorphism. 
\end{cor} 

\begin{proof} 
By~\cite[Theorem 4.23]{GLMY15}, the Hurewicz homomorphism $\pi_{1}(X)\xrightarrow{h'} H_1(X)$ induces an isomorphism 
$\pi_1(X)_{ab}\cong H_{1}(X)$. As $\overline{\pi}_{1}(X)\cong\pi_{1}(X)$, the definition of $h''$ implies that it induces an isomorphism $\overline{\pi}_1(X)_{ab}\cong H_{1}(X)$. By Proposition~\ref{GLMYHurewicz}, $\widetilde{H}_1$ has the same image as $h''$, so $\overline{\pi}_{1}(X)\xrightarrow{\widetilde{H}_1} H_1(X)$ is onto and the image may be regarded as $\overline{\pi}_1(X)_{ab}$. Hence $\overline{\pi}_{1}(X)_{ab}\xrightarrow{(\widetilde{H}_1)_{ab}} H_1(X)$ 
is an isomorphism. 
\end{proof}

\section{The existence of higher digraph homotopy groups} 
In this section we prove Theorem~\ref{introexistence}. 
Let $G$ be a based digraph and suppose that $\pi_{1}(G)$ is abelian. By Corollary~\ref{pi1abelian} there is an isomorphism $\overline{\pi}_{1}(G)\stackrel{\widetilde{H}_1}{\rightarrow} H_{1}(G)$. Iterating the commutative diagram in Theorem~\ref{pcompa} then gives a commutative diagram
$$\xymatrix{
     \overline{\pi}_{1}(G)\ar[r]^-{E_1}\ar[d]_{\cong}^{\widetilde{H}_1} & \overline{\pi}_{2}(\Sigma G)\ar[r]^-{E_2}\ar[d]^{\widetilde{H}_2}
         & \cdots\ar[r]^-{E_{n-1}} & \overline{\pi}_{n}(\Sigma^{n-1} G)\ar[d]^{\widetilde{H}_n} \\
     H_{1}(G)\ar[r]^-{E'_1} & H_{2}(\Sigma G)\ar[r]^-{E'_2} & \cdots\ar[r]^-{E'_{n-1}} & H_{n}(\Sigma^{n-1} G) }$$ 
where $\Sigma^i G$ is defined recursively for $i\geq 2$ by $\Sigma^i G=\Sigma(\Sigma^{i-1} G)$. 
By Proposition~\ref{susiso}, each GLMY suspension $E'_1,\ldots,E'_{n-1}$ is an isomorphism. Therefore going around the diagram in the anti-clockwise direction, each of the maps is an isomorphism. The commutativity of the diagram therefore implies that the composition of the maps in the clockwise direction is an isomorphism. Consequently, $\overline{\pi}_{1}(G)$
retracts off $\overline{\pi}_{n}(\Sigma^{n-1} G)$. Thus
$$\overline{\pi}_{n}(\Sigma^{n-1} G)\cong\overline{\pi}_{1}(G)\oplus A$$ for some abelian group $A$. 

We will use this to show the existence of nontrivial higher digraph homotopy groups. To do so, it is helpful to use the isomorphism $\overline{\pi}_1(G)\cong H_1(G)$ to rewrite the splitting of $\overline{\pi}_n(\Sigma^{n-1} G)$ as  
\begin{equation}\label{H1splitsoff} 
\overline{\pi}_{n}(\Sigma^{n-1} G)\cong H_{1}(G)\oplus A. 
\end{equation} 
In particular, any group that appears as a direct summand of $H_{1}(G)$ also appears as a direct summand of $\overline{\pi}_{n}(\Sigma^{n-1} G)$ for any $n\geq 2$. This leads to existence results for higher digraph homotopy groups. 

For example, to this point there is only one known nontrivial higher digraph homotopy group. The example in~\cite[Example 4.22]{LWYZ24} shows that there is a digraph $G$ such that $\overline{\pi}_{2}(G)\neq 0$, although the group $\overline{\pi}_{2}(G)$ is not explicitly identified. We prove that for each $n\geq 1$ there is a digraph whose $n^{th}$-digraph homotopy group is nontrivial (in fact, it has a $\mathbb{Z}$-summand). 

\begin{prop}\label{existenceZ} 
For every $n\geq 1$ there is a digraph $G_n$ such that $\overline{\pi}_{n}(G_n)$ has a $\mathbb{Z}$-summand. 
\end{prop} 

\begin{proof} 
By~\cite[Example 2.8]{GLMY15}, if $G_1$ is the digraph $0\rightarrow 1\rightarrow 2\rightarrow 3\rightarrow 0$ then $H_1(G)\cong\mathbb{Z}$.
Thus if $G_{i}=\Sigma G_{i-1}$ for each $2\leq i\leq n$, then~(\ref{H1splitsoff}) implies that $\overline{\pi}_{n}(G_n)$ has $\mathbb{Z}$ as a summand. 
\end{proof} 

In fact, by~\cite[Example 2.8]{GLMY15}, there is a large family of cycle digraphs (of which the graph $G_1$ used in the proof of Proposition~\ref{existenceZ} is one) having first digraph homology group isomorphic to $\mathbb{Z}$. Any of these could be used in place of $G_1$ to define new examples of digraphs $G_n$ with $\overline{\pi}_n(G_n)$ having a $\mathbb{Z}$-summand. 

Proposition~\ref{existenceZ} can also be extended to any finite torsion-free abelian group. 

\begin{prop}\label{existenceZk} 
For every $n\geq 1$ and $k\geq 1$ there is a digraph $G^k_n$ such that $\overline{\pi}_{n}(G^k_n)$ has a $\mathbb{Z}^{\oplus k}$-summand. 
\end{prop} 

\begin{proof} 
By~~\cite[Theorem 4.7]{GMY17}, there is a K\"{u}nneth formula for the box product in path homology. Thus if $G_1$ is the digraph used in the proof of Proposition~\ref{existenceZ}, then $H_1(G_1^{\Box k})\cong\mathbb{Z}^{\oplus k}$. Therefore, if $G^k_i=\Sigma G^k_{i-1}$ for $2\leq i\leq n$, then~(\ref{H1splitsoff}) implies that $\overline{\pi}_{n}(G^k_n)$ has $\mathbb{Z}^{\oplus k}$ as a summand. 
\end{proof} 

There is also torsion in higher digraph homotopy groups. 

\begin{prop}\label{existencetorsion} 
For every $n\geq 1$ there are digraphs $G_n(2)$ and $G_n(3)$ such that $\overline{\pi}_{n}(G_n(2))$ has a $\mathbb{Z}/2\mathbb{Z}$-summand and $\overline{\pi}_{n}(G_n(3))$ has a $\mathbb{Z}/3\mathbb{Z}$-summand. 
\end{prop} 

\begin{proof} 
By~\cite[Figure 5]{CHY22}, there are digraphs $G(2)$ and $G(3)$ such that $H_1(G_2)\cong\mathbb{Z}/2\mathbb{Z}$ and $H_1(G(3))\cong\mathbb{Z}/3\mathbb{Z}$. Therefore, if $G_1(2)=G(2)$ and $G_1(3)=G(3)$, and we define $G_i(2)=\Sigma G_{i-1}(2)$ and $G_i(3)=\Sigma G_{i-1}(3)$ for $2\leq i\leq n$, then~(\ref{H1splitsoff}) implies that $\overline{\pi}_{n}(G_n(2))$ has $\mathbb{Z}/2\mathbb{Z}$ as a summand and $\overline{\pi}_{n}(G_n(3))$ has $\mathbb{Z}/3\mathbb{Z}$ as a summand. 
\end{proof} 

\begin{proof}[Proof of Theorem~\ref{introexistence}] 
Combine Propositions~\ref{existenceZ}, \ref{existenceZk} and \ref{existencetorsion}. 
\end{proof} 

\begin{rem} 
In~\cite{CHY22} the authors conjecture that a certain iterated construction should produce a digraph $G(m)$ with $H_1(G(m))\cong\mathbb{Z}/m\mathbb{Z}$ for every $m\geq 2$. They claim the conjecture is true for $2\leq m\leq 8$ but only explicitly show it is true for $2\leq m\leq 3$. If the conjecture were true, then the same argument as for Proposition~\ref{existencetorsion} would give the existence of digraphs whose $n^{th}$-digraph homotopy group has a torsion summand of arbitrarily high order. 
\end{rem}

\end{document}